\documentclass{article}
\usepackage[utf8]{inputenc} 
\usepackage[T1]{fontenc}
\usepackage{lipsum}
\usepackage{fullpage}
\title{Genuine deformations of Euclidean hypersurfaces\\ in higher codimensions I}
\author{Diego N. Guajardo\footnote{Partially supported by CNPq and FAPERJ.}}
\date{}
\usepackage{graphicx}
\usepackage{amsmath}
\usepackage{amssymb}
\usepackage{pst-node}

\usepackage{tikz-cd} 

\usepackage{bigints}
\usepackage{relsize}
\usepackage{amsthm}
\usepackage{hyperref}
\hypersetup{
    colorlinks=true,
    linkcolor=blue,
    filecolor=magenta,      
    urlcolor=cyan,
    citecolor=blue,
}

\usepackage{tikz-cd}

\makeatletter
\newcommand{\tpitchfork}{%
  \vbox{
    \baselineskip\z@skip
    \lineskip-.52ex
    \lineskiplimit\maxdimen
    \m@th
    \ialign{##\crcr\hidewidth\smash{$-$}\hidewidth\crcr$\pitchfork$\crcr}
  }%
}
\makeatother

\usepackage[sorting=nyt]{biblatex} 

\DeclareFieldFormat[article, inbook]{title}{#1,} 
\DeclareFieldFormat[article]{pages}{#1}
\DeclareFieldFormat{journaltitle}{#1\isdot,}
\DeclareDelimFormat[bib]{nametitledelim}{: }
\renewbibmacro{in:}{}

\renewbibmacro*{issue+date}{%
  \ifboolexpr{test {\iffieldundef{year}}
              and test {\iffieldundef{issue}}}
    {}
    {\printtext[parens]{%
       \printfield{issue}%
       \setunit*{\addspace}%
       \usebibmacro{date}}}%
  \newunit}

\usepackage[shortlabels]{enumitem}

\newtheorem{thm}{Theorem}[section]

\newtheorem{lema}[thm]{Lemma}
\newtheorem{prop}[thm]{Proposition}
\newtheorem{cor}[thm]{Corollary}

\newtheorem{defn}[thm]{Definition} 
\theoremstyle{remark}
\newtheorem{remark}[thm]{Remark}

\providecommand{\keywords}[1]
{
  \small	
  \textbf{\textit{Keywords---}} #1
}

\begin{document}

\maketitle
\begin{abstract}
    Sbrana and Cartan gave local classifications for the set of Euclidean hypersurfaces $M^n\subseteq\mathbb{R}^{n+1}$ which admit another genuine isometric immersions in $\mathbb{R}^{n+1}$ for $n\geq 3$. The main goal of this paper is to extend their classification to higher codimensions. Our main result is a complete description of the moduli space of genuine deformations of generic hypersurfaces of rank $(p+1)$ in $\mathbb{R}^{n+p}$ for $p\leq n-2$. As a consequence, we obtain an analogous classification to the ones given by Sbrana and Cartan providing all local isometric immersions in $\mathbb{R}^{n+2}$ of a generic hypersurface $M^n\subseteq\mathbb{R}^{n+1}$ for $n\geq 4$. We also show how the techniques developed here can be used to study conformally flat Euclidean submanifolds. 
\end{abstract}
\hspace{30pt}
\keywords{Genuine rigidity, deformable submanifolds, conformally flat Euclidean submanifolds,

Darboux-Manakov-Zakharov systems.}
\section{Introduction}

The classical Theorem of Sacksteder \cite{Sack} states that a compact Euclidean hypersurface is rigid as long as the set of totally geodesic points doesn't disconnect the manifold. 
In \cite{FG} an analogous result for compact submanifolds $f:M^n\rightarrow\mathbb{R}^{n+p}$ and $g:M^n\rightarrow\mathbb{R}^{n+q}$ with $p+q<\min\{5,n\}$ is proved by allowing some natural and necessary singularities. This problem was studied before for $p=q=2$ in \cite{DG2}.

Locally, hypersurfaces are much more deformable. Sbrana in \cite{Sb} studied the local problem of classifying the Riemannian manifolds which possess at least two (locally) non-congruent isometric immersions $f,g:M^n\rightarrow\mathbb{R}^{n+1}$. 
He proved that, if $M^n$ is nowhere flat, then $M^n$ belongs to one of four types. The two non-generic types, the {\it surface-like} and {\it ruled} ones, are highly deformable. 
In contrast, the manifolds belonging to the {\it continuous type} possess a continuous one-parameter family of such immersions, while the ones of the {\it discrete type} have exactly two. This description was given in terms of what is now called the Gauss parametrization which parametrizes the hypersurface in terms of its Gauss map and its support function. A few years latter, Cartan in \cite{Ca} gave an equivalent description in terms of envelopes of spheres. For a modern approach to the problem see \cite{DFT2}.

In this work we extend the Sbrana-Cartan classification to higher codimension. For this, we use the concept of genuine rigidity which extends the one of isometric rigidity. This notion was introduced in \cite{DF3} and extended in \cite{FG}, and is more adequate for the study of rigidity in higher codimensions; see for example \cite{DFT}, \cite{DG2} and \cite{FF}.

Generic hypersurfaces in the Sbrana-Cartan classification satisfy that both the Gauss map and the support function are solutions of the same linear hyperbolic or elliptic partial differential equation. In this work we will naturally associate to our problem a Darboux-Manakov-Zakharov (DMZ) system of PDEs which plays the role of such PDE. 
Darboux introduced such systems to study the problem of triply orthogonal system of surfaces, which was a hot topic during the 19th century, to the point that Bianchi \cite{Bi} wrote a 850 pages book on the subject. DMZ systems and the $n-$orthogonal system of hypersurfaces have gained attention more recently due to the strong relation with a $n$-dimensional generalization of the Euler equation in hydrodynamics, see \cite{DN} and~\cite{T}. 
 
Recall that $(u_0\ldots,u_p)$ is a {\it conjugate chart} of an immersed submanifold of the sphere $h:L^{p+1}\rightarrow\mathbb{S}^{n}$ if the associated Christoffel symbols satisfy $\Gamma_{ij}^k=0$ for distinct indices and $\alpha^h(\partial_{u_i},\partial_{u_j})=0$, where $\alpha^h$ is the second fundamental form of $h$. 
Equivalently, $h$ as a map in $\mathbb{R}^{n+1}$ is a solution of the DMZ system
$$(Q(h))_{ij}:=Q_{ij}(h)=\partial^2_{ij}h-\Gamma_{ji}^i\partial_ih-\Gamma_{ij}^j\partial_jh+g_{ij}h=0, \quad \forall 0\leq i< j\leq p.$$
Notice the similarity with Cartan submanifolds; see for example \cite{KT} and \cite{KT2}. 

The work done by Dajczer, Florit and Tojeiro in \cite{DFT2} and \cite{DFT} is particularly important for this paper, since several of the techniques developed here were inspired by it. 
In particular, in \cite{DFT} they classify the Euclidean hypersurfaces of rank $2$ (that is, the number of non-zero principal curvatures is exactly two) that have genuine deformations in $\mathbb{R}^{n+2}$. 

The following is the main result of this work, which for $p=1$ recovers the Sbrana-Cartan classification. For this, we have extended the notion of species that defines those families simply by measuring the trivial holonomy of what we call the {\it Sbrana bundle associated to $Q$}. 
We say that a hypersurface $f:M^n\rightarrow\mathbb{R}^{n+1}$ of rank $(p+1)<n$ is of {\it $r^{th}$-type} if the moduli space of genuine deformations $g:M^n\rightarrow\mathbb{R}^{n+p}$ is naturally a union of at most $(p+1)$ convex open subsets of $\mathbb{R}^r$. 
\begin{thm}\label{teorema introduccion}
    Let $f:M^n\rightarrow\mathbb{R}^{n+1}$ be a simply connected hypersurface of rank $(p+1)$, with $1\leq p\leq n-2$. If $p\geq 7$ assume in addition that $f$ is not $(n-p+2)$-ruled. Then $f$ is genuinely rigid on $\mathbb{R}^{n+q}$ for any $q<p$. Moreover, if $f$ possesses a genuine deformation in $\mathbb{R}^{n+p}$ and is generic, then, along each connected component of an open dense subset of $M^n$, $f$ is of $r^{th}$-type for some $r\in\{0,\ldots,p\}$. In this case, the Gauss map $h$ of $f$ has a unique conjugate chart of $(p+1-r)^{th}$-species, and its support function $\gamma=\langle f,h\rangle$ also satisfies $Q(\gamma)=0$.
    
    Conversely, under the Gauss parametrization, $(h,\gamma)$ as above gives rise to an Euclidean hypersurface genuinely deformable in codimension $p$. Furthermore, $f$ is of $r^{th}$-type where $M^n$ is generic.
\end{thm}
We point out that in the converse the deformations may be in some semi-Euclidean space $\mathbb{R}^{n+p}_{\mu}$, that is, $\mathbb{R}^{n+p}$ with a non-degenerate inner product of index $\mu\leq p$. The value of $\mu$ is easily determined also by the trivial holonomy of the Sbrana bundle of $Q$. 

Although the Sbrana-Cartan work was done in 1908, it took almost a century to find explicit examples of hypersurfaces of the discrete type. The first examples, which are now called of {\it intersection type}, were found also in \cite{DFT2} as intersection of two generic flat hypersurfaces $N^{n+1}_1, N^{n+1}_2\subseteq\mathbb{R}^{n+2}$, in which case $Q$ is hyperbolic. 
This construction also shows the local nature of the classification by producing examples of connected locally deformable hypersurfaces of locally different types in the Sbrana-Cartan classification. Those examples are characterized by the vanishing of one of the Laplace invariants of $Q$.
Later, Dajczer-Florit in \cite{DF} gave a procedure to obtain the first examples of locally deformable hypersurfaces of discrete-type with $Q$ elliptic.

Until now there is no analogous classification to that of Sbrana and Cartan in higher codimensions, only classifications in certain restricted cases, not even in codimension $2$. In this case, Theorem 1 of \cite{DF3} shows that if $f:M^n\rightarrow\mathbb{R}^{n+1}$ is genuinely deformable in $\mathbb{R}^{n+2}$, then its rank must be at most three. If its rank is one or less the hypersurface is flat, and all its isometric immersions in $\mathbb{R}^{n+2}$ are described in Corollary 18 of \cite{FF}. Theorem 1 of \cite{DFT} describes the rank two generic case in terms of their support function $\gamma$ and a conjugate coordinate system for its Gauss map $h:L^2\rightarrow\mathbb{S}^n$, just as in Theorem \ref{teorema introduccion}. Moreover, it computes the moduli space $\mathcal{C}_h$ of deformations of $f$ in $\mathbb{R}^{n+2}$. Theorem \ref{teorema introduccion} for $p=2$ analyzes the generic rank three case.
Thus, the following result summarizes the above discussion, and characterizes all generic Euclidean hypersurfaces which are genuinely deformable in $\mathbb{R}^{n+2}$  and the respective moduli space of their {\it honest} deformations, as defined in \cite{FF}. The concept of honest rigidity is the natural one for such a result and is slightly stronger than genuine rigidity.
We point out that Theorem 1 of \cite{DFT} has a gap for hypersurfaces of intersection type. Yet, Theorem 33 of \cite{FF} and an adaptation of that result for Lorentz ambient space (Theorem \ref{deforamciones de las intersecciones en lorentz} bellow) allow us to fill this gap, describing the honest deformations for hypersurfaces of intersection type in codimension $2$ in terms of its {\it shared dimension} $I$; see Section \ref{seccion deformaciones en codim 2}. 

\begin{thm}\label{Descripcion SC p=2 generica}
    Let $f:M^{n}\rightarrow\mathbb{R}^{n+1}$ be a genuinely deformable hypersurface in codimension $2$. Then the rank of $M^n$ is at most 3. Assume that $M^n$ is generic and nowhere flat, in particular $n\geq 4$. Then each connected component $U$ of an open dense subset of $M^n$ falls in exactly one of these categories:
    \begin{enumerate}
        \item The rank of $U$ is $3$. The Gauss map $h:L^3\rightarrow\mathbb{S}^{n}$ is of $(3-r)^{th}$-species for some $r\in\{0,1,2\}$ and the support function $\gamma$ satisfying $Q(\gamma)=0$. In this case, $f|_U$ is of $r^{th}$-type and all its genuine deformations in $\mathbb{R}^{n+2}$ are honest deformations;
        \item The rank of $U$ is $2$ and $f|_U$ is not a Sbrana-Cartan hypersurface of intersection type. Then the Gauss map $h:L^2\rightarrow\mathbb{S}^{n}$ of $f|_U$ has a conjugate chart and the support function $\gamma$ satisfies $Q(\gamma)=0$. In this case, the moduli space of honest deformations is naturally $\mathcal{C}_h$;
        \item The rank of $U$ is $2$ and $f|_U$ is a Sbrana-Cartan hypersurface of intersection type. That is, $U$ is obtained as an intersection of two flat Riemannian hypersurfaces on $\mathbb{R}^{n+2}_{\nu}$ for $\nu\leq 1$ and $f|_U$ is the inclusion in one of such hypersurfaces.
        Then $f|_U$ is honestly rigid in $\mathbb{R}^{n+2}$, unless $I=2$. In the latter case, the moduli space of honest deformations of $f$ in $\mathbb{R}^{n+2}$ is naturally an open interval of $\mathbb{R}$.
    \end{enumerate}
\end{thm}

\text{ }

The study of conformally flat Euclidean submanifolds in codimension $2$, namely, submanifolds $M^n\subseteq\mathbb{R}^{n+2}$ which are conformally flat, is strongly linked to the Sbrana-Cartan theory. In fact, the description given in \cite{DF2} for such submanifolds is similar to the one given for deformable hypersurfaces, and some examples can be found using intersections of flat submanifolds in a similar way as for deformable hypersurfaces; see \cite{DF6}. However, in this case we must consider Riemannian hypersurfaces of the Lorentz space. This and the development of the proof of Theorem~\ref{teorema introduccion} led us to consider hypersurfaces and its genuine deformations in semi-Euclidean spaces.

It is therefore not surprising that the techniques developed in this work can be used also to study conformally flat submanifolds $g:M^n\rightarrow\mathbb{R}^{n+p+1}$. 
As proven in \cite{DF2}, if $p\leq n-4$, (locally) such manifolds $M^n$ can be obtained as the intersection of some Riemannian hypersurface $F:N^{n+1}\rightarrow\mathbb{R}_1^{n+2}$ with the light cone, and $N^{n+1}$ admits an isometric immersion $G:N^{n+1}\rightarrow\mathbb{R}^{n+p+1}$ such that $g=G|_{M^n}$. 
The hypersurface $F$ must have rank at most $(p+1)$. 
The following result characterizes such Riemannian hypersurfaces of rank $(p+1)$. 
This generalizes Theorem 5 of \cite{DF2} that deals with the case $p=1$. 
As before, the hypothesis of being generic is to discard the surface-like situation, and for the converse the deformations may be in some semi-Euclidean space $\mathbb{R}^{m+p}_{\mu}$.

\begin{thm}\label{teorema de hipersuperficies lorentzianas}
    Let $F:N^{m}\rightarrow \mathbb{R}_1^{m+1}$ be a Riemannian hypersurface of rank $(p+1)\geq2$. Then $N^{m}$ cannot be isometrically immersed in $\mathbb{R}^{m+q}$ for any $q<p$. Assume further that there exists an isometric immersion $G:N^m\rightarrow\mathbb{R}^{m+p}$. Then, the Gauss map $h$ of $F$ has a unique conjugate chart of the $k^{th}-$species for some $k\in\{1,\ldots,p+1\}$, and the support function $\gamma=\langle f,h\rangle$ also satisfies $Q(\gamma)=0$.

    Conversely, under the Gauss parametrization, $(h,\gamma)$ as above gives rise to an Riemannian hypersurface $F$ deformable in codimension $p$. Furthermore, if $N^{m}$ is generic, then $F$ is of $(p+1-k)^{th}$-type.
\end{thm}

In \cite{YoGD2}, the sequel of this paper, we will provide examples of the hypersurfaces described in this work using the intersection techniques developed in \cite{DFT2}. In addition, we will present an analogous result to Theorem \ref{teorema introduccion} classifying the genuine deformations of Euclidean hypersurfaces of rank $(p+1)$ in $\mathbb{R}^{n+p+1}$, generalizing Theorem~1 in \cite{DFT} to higher codimensions. 

There are several results in the literature which are described in terms of surfaces with conjugate charts, and in several of them this surface is the leaf space of some umbilical distribution of codimension 2; besides the ones already cited, see for example \cite{CThiperConf2}, \cite{DJV}, \cite{DT2},  \cite{DV}, \cite{DV2}, \cite{DV3}. We believe that some of those results can be extended to dimensions bigger that 2 using the tools developed in this paper.

This paper is organized as follows. In Section \ref{preliminaries} we recall the notions of genuine rigidity, Gauss parametrization, DMZ systems, among others. Section \ref{seccion descripcion del problema} is devoted to describe the rigidity problem and to prove Theorem \ref{teorema introduccion}. In Section \ref{seccion deformaciones en codim 2} we demonstrate Theorem \ref{Descripcion SC p=2 generica}, while in Section \ref{Seccion lorentzianas} we analyze the conformal case and prove Theorem~\ref{teorema de hipersuperficies lorentzianas}. We end our work with an Appendix with auxiliary results.

\text{ }

{\it Acknowledgment.} This work is a portion of the author's Ph.D. thesis at IMPA - Rio de Janeiro. The author would like to thank his adviser, Prof. Luis Florit for his orientation. 

\section{Preliminaries}\label{preliminaries}

Several of the tensors that we deal with in this work are more easily treatable in $(TM)_\mathbb{C}$,
the complexification of the tangent bundle of some manifold $M^n$. In order to do this, we need to establish some identifications.

Given a (finite dimensional) real vector space $\mathbb{W}$ we denote by $\mathbb{W}_\mathbb{C}=\mathbb{W}\otimes\mathbb{C}$ its complexification. Conversely, let $\mathbb{V}$ be a complex vector space with an antilinear map $C:\mathbb{V}\rightarrow \mathbb{V}$, that is, $C(\lambda v)=\overline{\lambda}C(v)$ for $\lambda\in\mathbb{C}$, satisfying $C^2=\text{Id}$. Define $\text{Re}(\mathbb{V})=\text{Re}_C(\mathbb{V})=\{v\in \mathbb{V}:Cv=v\}$ and $\text{Im}(\mathbb{V})=\{v\in \mathbb{V}:Cv=-v\}$. We have that $i:\text{Re}(\mathbb{V})\rightarrow\text{Im}(\mathbb{V})$, $i(v)=iv$ is a real isomorphism, so $\dim_\mathbb{R}(\text{Re}(\mathbb{V}))=\dim_\mathbb{C}(\mathbb{V})$, since $\mathbb{V}=\text{Re}(\mathbb{V})\oplus\text{Im}(\mathbb{V})$ as real vector spaces. The map $C$ is called a {\it conjugation map}. 
Notice that $\mathbb{W}_{\mathbb{C}}$ comes with its natural conjugation $v+iw\rightarrow\overline{v+iw}:=v-iw$ for $v,w\in\mathbb{W}$.

Consider a complex basis $\{e_i\}_{i\in I}$ of $\mathbb{W}_{\mathbb{C}}$ closed under the conjugation, that is, for any index $i\in I$ there is a unique index $\overline{i}\in I$ such that $\overline{e_i}=e_{\overline{i}}$. The $\mathbb{C}$-antilinear map defined by $C(e_i)=e_{\overline{i}}$ is the natural conjugation and satisfies that $\mathbb{W}=\text{Re}_C(\mathbb{W}_{\mathbb{C}})$. Hence any tensor in $\mathbb{W}_{\mathbb{C}}$ with the natural compatibility condition with respect to this basis automatically corresponds to a real tensor in $\mathbb{W}$. 

\subsection{Flat bilinear forms}
Given a bilinear map $\beta:\mathbb{V}\times \mathbb{U}\rightarrow \mathbb{W}$ between real vector spaces, set
\begin{equation*}
    \mathcal{S}(\beta)=\text{span}\{\beta(X,Y):X\in \mathbb{V}, Y\in \mathbb{U}\}\subseteq \mathbb{W}.
\end{equation*}
The (left) \emph{nullity} of $\beta$ is the vector subspace
\begin{equation*}
    \Delta_\beta=\mathcal{N}(\beta)=\{X\in \mathbb{V}:\beta(X,Y)=0\, ,\,\forall Y\in \mathbb{U}\}\subseteq \mathbb{V}.
\end{equation*}
For each $Y\in \mathbb{U}$ we denote by $\beta^Y:\mathbb{V}\rightarrow \mathbb{W}$ the linear map $\beta^Y(X)=\beta(X,Y)$. Let
\begin{equation*}
    \text{Re}(\beta)=\{Y\in \mathbb{U}:\dim(\text{Im}(\beta^Y))\text{ is maximal}\}
\end{equation*}
be the set of (right) \emph{regular elements of} $\beta$, and set $i(\beta):=\dim(\text{Im}(\beta^Y))$ for any $Y\in\text{Re}(\beta)$. The set of regular elements is open and dense in $\mathbb{V}$. There are obvious definitions for left regular elements and right nullity.

Assume now that $\mathbb{W}$ has a non-degenerate inner product $\langle\cdot,\cdot\rangle:\mathbb{W}\times \mathbb{W}\rightarrow\mathbb{R}$. We denote $\mathbb{W}^{p,q}$ to point out that the inner product in $\mathbb{W}$ has signature $(p,q)$. We say that $\beta$ is $\mathit{flat}$ if
\begin{equation*}
    \langle\beta(X,Y),\beta(Z,W)\rangle=\langle\beta(X,W),\beta(Z,Y)\rangle\quad\forall X,Z\in \mathbb{V}\quad \forall Y,W\in \mathbb{U}.
\end{equation*}

For a symmetric bilinear map $\beta:\mathbb{V}\times\mathbb{V}\rightarrow\mathbb{W}$, we say that $\beta$ {\it diagonalizes} if there exists a basis $\{X_i\}_{i}$ of $\mathbb{V}_\mathbb{C}$ such that $\{X_i\}_{i}=\{\overline{X_i}\}_{i}$ and $\beta(X_i,X_j)=0$ for all $i\neq j$, where we are extending $\beta$ by $\mathbb{C}$-bilinearity $\beta:\mathbb{V}_\mathbb{C}\times \mathbb{V}_\mathbb{C}\rightarrow \mathbb{W}_\mathbb{C}$. We denote $\overline{j}$ the index such that $\overline{X_j}=X_{\overline{j}}$. 

There are two results that we need in order to bound the dimension of the nullity of a flat bilinear form. The first one due to Moore \cite{Moo} is valid for non-necessarily symmetric ones.
\begin{lema}\label{nulidad para no simetrica}
    Let $\beta:\mathbb{V}\times \mathbb{U}\rightarrow \mathbb{W}$ be a flat bilinear form. If $X\in \mathbb{U}$ is a right regular element, then
    \begin{equation*}
        \mathcal{S}(\beta|_{\ker(\beta^X)\times \mathbb{U}})\subseteq \beta^X(\mathbb{V})\cap\beta^X(\mathbb{V})^\perp.
    \end{equation*}
    In particular, if $\beta^X(\mathbb{V})$ is non-degenerate then $\Delta_\beta=\ker(\beta^X)$ and
    \begin{equation*}
        \dim(\Delta_\beta)\geq \dim(\mathbb{V})-\dim(\text{Im}(\beta^X)).
    \end{equation*}
\end{lema}

The second result proved in \cite{DF4} is only valid for symmetric flat bilinear forms and is called the Main Lemma in the literature.
\begin{lema}[Main Lemma]\label{Main lemma}
    Let $\beta:\mathbb{V}^n\times \mathbb{V}^n\rightarrow \mathbb{W}^{p,q}$ be a flat symmetric bilinear form such that $\mathcal{S}(\beta)=\mathbb{W}^{p,q}$. If $\min\{p,q\}\leq 5$ then
    \begin{equation*}
        \dim(\Delta_\beta)\geq n-p-q.
    \end{equation*}
\end{lema}
We point out that the proof given in \cite{DF4} has a gap for $\min\{p,q\}=6$, in which case there are counterexamples as shown in \cite{DF5}. The correct statement for this case was given in \cite{DF3}.
\subsection{Genuine rigidity}\label{pares de inmersiones}

In this subsection we recall the notion of genuine rigidity which naturally extends the one of isometric rigidity, and that is more adequate to study deformation of hypersurfaces in higher codimensions.

\text{ }

Given a Riemannian manifold $M^n$ and $x\in M^n$, the {\it nullity of} $M^n$ {\it at} $x$ is the nullity of the curvature tensor $R$ of $M^n$, that is, the subspace of $T_xM$ given by
\begin{equation*}
    \Gamma(x)=\mathcal{N}(R_x)=\{X\in T_xM: R(X,Y)Z=0,\forall Y,Z\in T_xM\}.
\end{equation*}
The {\it rank of} $M^n$ {\it at} $x$ is defined by $n-\mu$, where $\mu=\dim(\Gamma(x))$. As the results that we are looking for are of local nature and our subspaces are all either kernels or images of smooth tensor fields, we will always  work on each connected component of an open dense subset of $M^n$ where all these dimensions are constant and thus all the subbundles are smooth without further notice. In particular, we assume that $\mu$ is constant and hence the second Bianchi identity implies that $\Gamma$ is a totally geodesic distribution, namely, $\nabla_{\Gamma}\Gamma\subseteq\Gamma$.

For an isometric immersion $f:M^n\rightarrow\mathbb{R}^{n+q}$ we denote by $\alpha^f:TM\times TM\rightarrow T^\perp _fM$ its second fundamental form. We define the $\mathit{relative}$ \emph{nullity of} $f$ \emph{at} $x$ as $\Delta_f(x):=\mathcal{N}(\alpha^f_x)$ and the \emph{rank of} $f$ as $n-\nu_f$, where $\nu_f=\dim(\Delta_f)$. Notice that if $f$ is a hypersurface then $\mu=\nu_f\leq n-2$ outside of the flat points of $M^n$. 

Given two isometric immersions $f:M^n\rightarrow\mathbb{R}^{n+q}$ and $g:M^n\rightarrow\mathbb{R}^{n+p}$, it is useful to work with the vector bundle $W=T^\perp_g M\oplus T^\perp_f M$, in which we define the semi-Riemannian metric with signature $(p,q)$ given by
\begin{equation*}
    \langle(\xi_1,\eta_1),(\xi_2,\eta_2)\rangle=\langle\xi_1,\xi_2\rangle_{T^{\perp}_gM}-\langle\xi_1,\xi_2\rangle_{T^{\perp}_fM}.
\end{equation*}
The bilinear tensor $\beta=(\alpha^g,\alpha^f):TM\times TM\rightarrow W$ is flat with respect to this metric by the Gauss equations of $f$ and $g$.
We also have the compatible connection in $W$ induced by the normal connections $\hat{\nabla}:=(\nabla^{\perp,g},\nabla^{\perp,f})$.
From the Codazzi equations for $f$ and $g$, $\beta$ is a Codazzi tensor, i.e. it satisfies
\begin{equation*}
    (\hat{\nabla}_X\beta)(Y,Z)=(\hat{\nabla}_Y\beta)(X,Z),\quad\forall X,Y,Z\in TM.
\end{equation*}
In particular, if in the above equation we take $X,Z\in\Delta_{\beta}=\Delta_f\cap \Delta_g$, we conclude that $\nabla_X Z\in\Delta_\beta$, that is, $\Delta_\beta$ is integrable and totally geodesic. 

We say that the pair $\{f,g\}$ {\it extends isometrically} if there exists a Riemannian manifold $N^{n+r}$, an isometric embedding $j:M^n\rightarrow N^{n+r}$ and two isometric immersions $F:N^{n+r}\rightarrow\mathbb{R}^{n+q}$, $G:N^{n+r}\rightarrow\mathbb{R}^{n+q}$ such that $f=F\circ j$ and $g=G\circ j$. That is, the following diagram commutes:
\begin{center}
    \begin{tikzcd}[ arrows={-stealth}]
        & & \mathbb{R}^{n+p}\\
        \quad\quad M^n \arrow{urr}{g}\arrow[swap]{drr}{f}     \arrow[hook]{r}{j} %
        & N^{n+r} \arrow[swap]{ur}{G}\arrow{dr}{F} \\
        & & \mathbb{R}^{n+q}
    \end{tikzcd}
\end{center}
We say that the pair $\{f,g\}$ is {\it genuine}, or that $g$ is a {\it genuine deformation of $f$} when $f$ is fixed, if there is no open subset $U\subseteq M$ such that $\{f|_U,g|_U\}$ extends isometrically. An isometric immersion $f:M^n\rightarrow\mathbb{R}^{n+q}$ is said to be {\it genuinely rigid} in $\mathbb{R}^{n+p}$ if there is no open subset $U\subseteq M^n$ such that $f|_U$ admits a genuine deformation in $\mathbb{R}^{n+p}$.
If that is not the case, we say that $f$ is {\it genuinely deformable} in $\mathbb{R}^{n+p}$.
In particular, when $f$ is a hypersurface, that $g:M^n\rightarrow\mathbb{R}^{n+p}$ is a genuine deformation of $f$ means that there is no open subset $U\subseteq M^n$ such that $g|_U=h\circ f|_U$, where $h:V\subseteq\mathbb{R}^{n+1}\rightarrow\mathbb{R}^{n+p}$ is some isometric immersion of an open subset $V$ with $f(U)\subseteq V$. 

We say that $f:M^n\rightarrow\mathbb{R}^{n+q}$ is $R^d${\it -ruled} (or $d${\it-ruled}), if $R^d\subseteq TM$ is a $d$-dimensional totally geodesic distribution whose leaves are mapped by $f$ to (open subsesets of) affine subspaces of $\mathbb{R}^{n+q}$. 
Theorem $1$ of \cite{DF3} says that a genuine pair $f:M^n\rightarrow\mathbb{R}^{n+q}$ and $g:M^n\rightarrow\mathbb{R}^{n+p}$ with $\min\{p,q\}\leq 6$ must be mutually $R^d$-ruled, this ruling contains $\Delta_\beta$, and it gives a sharp estimate for $d$. 

Let $(p+1)$ be the rank of a nowhere flat hypersurface $f:M^n\rightarrow\mathbb{R}^{n+1}$. Theorem $1$ of \cite{DF3} shows that, if $f$ is not $(n-p+3)$-ruled then $f$ is genuinely rigid in $\mathbb{R}^{n+q}$ for all $q<p$. Notice that the condition of not being $(n-p+3)$-ruled is trivially satisfied for $p\leq 6$ by the following elementary fact.
\begin{lema}\label{AR,R=0}
    Let $A:\mathbb{R}^n\rightarrow\mathbb{R}^n$ be a linear and symmetric map with respect to the Euclidean inner product. If there exists a $d$-dimensional subspace $R\subseteq\mathbb{R}^n$ such that $\langle A(R),R\rangle=0$, then $rank(A)\leq 2(n-d)$.
\end{lema}
Therefore it is natural to study genuine deformations of hypersurfaces of rank $(p+1)$ in $\mathbb{R}^{n+p}$.
Consider thus $g:M^n\rightarrow\mathbb{R}^{n+p}$ a genuine deformation of such an $f:M^n\rightarrow\mathbb{R}^{n+1}$. Let $\beta=\alpha^g\oplus\alpha^f$ and assume that $\mathcal{S}(\beta)$ is non-degenerate (this will be our case by Proposition \ref{S(beta) definido si no es reglada}). By the Main Lemma we have
$$n-p-1\leq n-\dim S(\beta)\leq \dim(\Delta_\beta)\leq\nu_f=n-p-1.$$ 
Hence, $S(\beta)=W^{p,1}$ and $\Delta_\beta=\Delta_f=\Gamma$. In particular, $\Delta_g\subseteq\Gamma=\Delta_\beta\subseteq\Delta_g$. We conclude that
\begin{equation*}
    \Delta_f=\Delta_g=\Delta_\beta=\Gamma.
\end{equation*}

All the definitions of this subsection have their natural extensions to the semi-Riemannian context, and we will use them without further mention.

\subsection{The Gauss parametrization}\label{seccion parametrizacion de gauss}
An important step in our approach to characterize genuine deformations of hypersurfaces of rank $(p+1)$ is to reduce the problem to the quotient space of nullity leaves $\pi:M^n\rightarrow L^{p+1}=M/\Gamma$. Once this is done, we obtain a classification of the hypersurfaces themselves by means of the Gauss parametrization that we describe next. For a more detailed description see \cite{DG}. 

\text{ }

Let $f:M^n\rightarrow\mathbb{R}^{n+1}$ be an orientable Euclidean hypersurface with constant relative nullity $\nu_f$. If $\rho:M^n\rightarrow\mathbb{S}^n$ is the Gauss map of $f$, then $\rho$ is constant along the leaves of $\Delta_f$. Hence, there is $h:L=M/\Delta_f\rightarrow \mathbb{S}^n$, such that $\rho=h\circ\pi$. This map $h$ is in fact an immersion, so we always consider on $L$ the metric induced by $h$. To give a complete local description of $f$ in terms of $h$ it is necessary to consider also its {\it support function} $\gamma:L\rightarrow\mathbb{R}$, which is defined by $\gamma\circ\pi=\langle f,\rho\rangle$. From $h$ and $\gamma$ we can recover $f(M^n)$ locally using the {\it Gauss parametrization} given by $\psi:T^\perp_hL\rightarrow\mathbb{R}^{n+1}$,
\begin{equation}\label{eq:parametrizacion de gauss}
    \psi(x,w)=(\gamma h+\nabla\gamma)(x)+w.
\end{equation}
We also denote the Gauss parametrization of $f$ simply by $(h,\gamma)$.
This useful tool was introduced by Sbrana in \cite{Sb} precisely to study rigidity of hypersurfaces of rank $2$, but since then it has had several applications in other contexts.

In particular, using the Gauss parametrization we have a local description of all flat hypersurfaces $f:M^n\rightarrow\mathbb{R}^{n+1}$. By the Gauss equation, the rank of $f$ is at most one. If $\nu_f=n$ then $f(M)$ is an open subset of some affine hyperplane. If $\nu_f=n-1$, then $f(M)$ can be (locally) described with a regular curve $h(s)$ in $\mathbb{S}^n$ and a real function $\gamma(s)$. A deeper analysis can be done to classify flat hypersurfaces in codimension two by means of a different parametrization. This was recently fully understood in Corollary 18 of \cite{FF}, and partially earlier in Theorem 13 of \cite{DF2}. In \cite{YoFlatE} they prove an analogous result for generic Euclidean flat submanifolds $M^n\subseteq \mathbb{R}^{n+p}$ and $p\leq n$.

\subsection{The Sbrana-Cartan classification}\label{subseccion SB clasificacion}
The Sbrana-Cartan classification gives a local description of all hypersurfaces of $f:M^n\rightarrow\mathbb{R}^{n+1}$ which possess genuine (namely, non-congruent) deformations in $\mathbb{R}^{n+1}$. To recall it we need a few definitions and results. 

\text{ }

By the classical Beez-Killing rigidity theorem, in order for $f:M^n\rightarrow\mathbb{R}^{n+1}$ to have a genuine deformation in $\mathbb{R}^{n+1}$ it must have rank at most $2$ everywhere. If the rank of $f$ is $1$ or $0$, then $M^n$ is flat and, as seen above, its genuine deformations can be easily understood by means of the Gauss parametrization. Hence, the interesting cases are among hypersurfaces of rank $2$. 

\begin{defn}
    \normalfont A hypersurface $f:M^n\rightarrow\mathbb{R}^{n+1}$ is called {\it surface-like} if there exists a surface $L^2\subseteq\mathbb{R}^3$ (resp. $L^2\subseteq\mathbb{S}^3$) such that $f(M^n)\subseteq L^2\times\mathbb{R}^{n-2}\subseteq\mathbb{R}^3\times\mathbb{R}^{n-2}$ (resp. $f(M^n)\subseteq C(L^2)\times\mathbb{R}^{n-3}\subseteq\mathbb{R}^4\times\mathbb{R}^{n-3}$ where $C(L^2)$ is the radial cone obtained from $L^2\subseteq\mathbb{S}^3$).
\end{defn}

In the Sbrana-Cartan classification the family of surface-like hypersurfaces is the first one among rank $2$ hypersurfaces which have genuine deformations.
Moreover, if $f$ as above is surface-like, then any genuine deformation of $f$ is given by a genuine deformation of $L^2$ in $\mathbb{R}^3$ (resp. in $\mathbb{S}^3$). However, a complete classification of the genuine deformations of surfaces is currently out of reach. 

The second family of genuinely deformable hypersurfaces of rank $2$ is that of $(n-1)$-ruled ones. It turns out that they all are highly deformable, any deformation preserves the rulings and the moduli space of genuine deformations is easily seen to be the set of smooth functions of one variable.

In order to describe the remaining  deformable hypersurfaces we need to recall some definitions.

\begin{defn}
    \normalfont Given a surface $h:L^2\rightarrow\mathbb{S}^n$, we call a coordinate system $(u,v)\in\mathbb{R}^2$ {\it real conjugate} if its second fundamental form satisfies $\alpha^h(\partial_u,\partial_v)=0$. Similarly, a coordinate system $z\in\mathbb{C}$ is called {\it complex conjugate} if $\alpha^h(\partial_u,\partial_{v})=0$, where $u=z=\overline{v}$. Accordingly, we say that $h$ is of {\it real} (resp. {\it complex}) {\it type}.
\end{defn}

Given a surface $h:L^2\rightarrow\mathbb{S}^n$ with a real (resp. complex) conjugate system $(u,v)$ and $\Gamma_{vu}^u, \Gamma_{uv}^v$ its Christoffel symbols, assume that the following system of PDE
\begin{equation}\label{sistema de edp Sbana cartan clasico}
    \left\{ \begin{array}{lr}
        \partial_u\tau=2\Gamma_{uv}^v\tau(1-\tau)\\
        \partial_v\tau=2\Gamma_{vu}^u(1-\tau),
    \end{array}\right.
\end{equation}
has a solution $\tau:L^2\rightarrow\mathbb{R}$ (resp. $\tau:L^2\rightarrow\mathbb{S}^1\subseteq\mathbb{C}$) other than the trivial one $\tau\equiv 1$. The integrability condition of this system is
\begin{equation}
    (\partial_v\Gamma_{uv}^v-2\Gamma_{uv}^v\Gamma_{vu}^v)\tau=\partial_u\Gamma_{vu}^u-2\Gamma_{uv}^u\Gamma_{vu}^v.
\end{equation}
Then $h$ is called of {\it first species} if the above equation is trivially satisfied, that is,
\begin{equation}
    \partial_u\Gamma_{vu}^u=2\Gamma_{uv}^u\Gamma_{vu}^v=\partial_v\Gamma_{uv}^v.
\end{equation}
We say that $h$ is of {\it second species} if $\partial_v\Gamma_{uv}^v\neq 2\Gamma_{uv}^v\Gamma_{vu}^v$, $\partial_u\Gamma_{vu}^u\neq 2\Gamma_{uv}^v\Gamma_{vu}^v$ and
\begin{equation}\label{solucion segunda especie sbrana cartan clasico}
    \tau=\frac{\partial_v\Gamma_{uv}^v-2\Gamma_{uv}^u\Gamma_{vu}^v}{\partial_u\Gamma_{vu}^u-2\Gamma_{uv}^u\Gamma_{vu}^v}\neq 1 
\end{equation}
is the necessarily unique solution of (\ref{sistema de edp Sbana cartan clasico}). For the real case, we also require that $\tau$ is positive.

\begin{thm}[Sbrana \cite{Sb}, Cartan \cite{Ca}]
    Let $f:M^n\rightarrow\mathbb{R}^{n+1}$ be a genuinely deformable hypersurface of rank $2$. Assume further that $f$ is nowhere surface-like nor $(n-1)$-ruled. Then, along connected components of an open dense subset, its Gauss map $h:L^2\rightarrow\mathbb{S}^n$ is of first or second species, and, with respect to its conjugate coordinate system, the support function satisfies 
    \begin{equation*}
        \partial^2_{uv}\gamma-\Gamma_{vu}^u\partial_u\gamma-\Gamma_{uv}^v\partial_v\gamma+\gamma g_{uv}=0.
    \end{equation*}
    If $h$ is of first species, then the moduli space of genuine deformations of $f$ is naturally parametrized by the positive initial conditions for $\tau$ solving (\ref{sistema de edp Sbana cartan clasico}). This set is $\mathbb{R}_{>0}\setminus\{1\}\cong\mathbb{R}\setminus\{0\}$ for the real type, while $\mathbb{S}^{1}\setminus\{1\}\cong\mathbb{R}$ for the complex type. If $h$ is of second species, the hypersurface $f$ has a unique genuine deformation.
\end{thm}

We say that a deformable hypersurface $f:M^n\rightarrow\mathbb{R}^{n+1}$ is of the {\it continuous type} (resp. {\it discrete type}) if it is described by the above theorem and the Gauss map is of the first species (resp. second species). 

\begin{remark}
    In the case that the Gauss map is of second species and real type but $\tau$ given by (\ref{solucion segunda especie sbrana cartan clasico}) is negative, we can associate an isometric immersion in the Lorentz space $\mathbb{R}^{n+1}_1$, as shown in Theorem 5 of \cite{DF2}. In a similar way, when the Gauss map is of the first species, for each initial condition for $\tau$ negative we can associate an isometric immersion $g=g_\tau:M^n\rightarrow\mathbb{R}_1^{n+1}$. This is an important result for studying conformally flat submanifolds and one of the main reasons we will not restrict ourselves only to Riemannian ambient Euclidean spaces.  
\end{remark}

\subsection{Darboux-Manakov-Zakharov (DMZ) systems}\label{sistemas de Darboux}
This subsection describes Darboux-Manakov-Zakharov (overdetermined) systems of PDEs. They have a crucial role in the description of our geometric problem.

\text{ }

One of Darboux many interests was that of orthogonal systems of coordinates for $\mathbb{R}^{p+1}$. That is, coordinate systems $(u_0,\ldots,u_p)$ of $\mathbb{R}^n$ such that the Euclidean metric is expressed as
\begin{equation*}
    ds^2=v_0^2du_0^2+\ldots+v_p^2du_p^2,
\end{equation*}
for some smooth functions $v_i=v_i(u_0,\ldots,u_p)$. For $p=2$ this problem is called the problem of triply orthogonal systems of surfaces. 
It is easy to verify that for such a coordinate system we have that, for three distinct indices, the Christoffel symbols satisfy $\Gamma_{ij}^k=0$ and $\Gamma_{ji}^i=\frac{\partial_jv_i}{v_i}$.
This naturally implies that for any indices $i\neq j<k\neq i$ we have that
\begin{equation}\label{ecuacion de los vi}
    \partial^2_{jk}v_i-\Gamma_{kj}^j\partial_jv_i-\Gamma_{jk}^k\partial_kv_i=0.
\end{equation}
Additional non-linear equations must be satisfied by the $v_i$'s in order to obtain a flat metric.

Darboux proposed an associated system of PDEs to find solutions of the last equations and linearize the problem. Consider $(u_0,\ldots,u_p)=(z_0,\overline{z_0},\ldots, z_{s-1},\overline{z_{s-1}},x_{2s},\ldots,x_p)\in\mathbb{C}^{2s}\times\mathbb{R}^{p+1-2s}$ for some $s$, and denote by $\overline{i}$ the unique index which satisfies $\overline{u_i}=u_{\overline{i}}$. 
The collection $Q=(Q_{ij})_{i<j}$ of second order linear PDEs given by
\begin{equation}\label{Q}
    (Q(\xi))_{ij}=Q_{ij}(\xi)=\partial^2_{ij}\xi+a_{ij}^j\partial_j\xi+a_{ji}^i\partial_i\xi+b_{ij}\xi=0\quad\forall\, 0\leq i<j\leq p,
\end{equation}
for $\partial_i=\partial_{u_i}$, and some smooth complex functions $a_{ij}^j,b_{ij}$ satisfying $\overline{a_{ij}^j}=a_{\overline{i}\,\overline{j}}^{\overline{j}}$, $\overline{b_{ij}}=b_{\overline{i}\,\overline{j}}$ is called a {\it Darboux-Manakov-Zakharov (DMZ) system}. Darboux only analyzed the case when $s=0$ and $p=2$, but this generalization is natural and is needed for this work. Notice the similarity between (\ref{ecuacion de los vi}) and (\ref{Q}) with $b_{ij}=0$ (for us the case $b_{ij}=0$ is irrelevant, see Proposition \ref{carta conjugada esfera=carta conjugada euclideano}).

As shown in \cite{KT}, we can associate a set of {\it Laplace invariants} to a DMZ system. Those invariants determine the system when all the equations are hyperbolic as shown in Theorem 1 of \cite{KT2}. They are defined for distinct indices by 
$$m_{ij}=\partial_ia_{ji}^i+a_{ji}^ia_{ij}^j-b_{ij},$$
$$m_{ijk}=a_{kj}^k-a_{ij}^i.$$

We now provide the natural generalization of the notion of conjugate chart for higher dimensional submanifolds.
\begin{defn}\label{definicion de carta conjugada}
    \normalfont A coordinate system $(z_0,\ldots z_{s-1},x_{2s},\ldots, x_{p})\in\mathbb{C}^s\times\mathbb{R}^{p+1-2s}$ of a submanifold $h:L^{p+1}\rightarrow\mathbb{S}^n\subseteq\mathbb{R}^{n+1}$ is called {\it conjugate} if $h$ is a solution of a DMZ system with respect to $(u_0,\ldots,u_{p})=(z_0,\overline{z_0},\ldots,z_{s-1},\overline{z_{s-1}},$
    $x_{2s},\ldots,x_{p})$, that is
    \begin{equation}\label{equacion de ser carta conjugada}
        Q_{ij}(h)=\partial^2_{ij}h-\Gamma_{ji}^i\partial_ih-\Gamma_{ij}^j\partial_jh+g_{ij}h=0, \quad \forall i< j,
    \end{equation}
    where $\{\partial_i=\partial_{u_i}\}_{i=0}^p$ is the local coordinate frame for $(TL)_{\mathbb{C}}$, $\Gamma_{ji}^i,\Gamma_{ij}^j:L^{p+1}\rightarrow\mathbb{C}$ are necessarily the Christoffel symbols associated to this frame, and $g_{ij}=\langle\partial_ih,\partial_jh\rangle.$ 
\end{defn}
\begin{remark}\label{condicion integrabilidad carta conjugada}
    Notice that (\ref{equacion de ser carta conjugada}) is equivalent to $\alpha^h(\partial_i,\partial_j)=0$ and $\Gamma_{ij}^k=0$ for distinct indices. Then the Gauss equation of $h$ for three distinct indices becomes
    $$R(\partial_i,\partial_j)\partial_k=g_{jk}\partial_i-g_{ik}\partial_j,$$
    which is equivalent to 
    \begin{equation}\label{integrabilidad definicion de conjugada}
        \partial_i\Gamma^j_{kj}+\Gamma^j_{kj}\Gamma^j_{ij}-\Gamma^j_{kj}\Gamma^k_{ik}-\Gamma^j_{ij}\Gamma^i_{ki}+g_{ik}=0.
    \end{equation}
    Those equations and the compatibility of the metric with the connection are precisely the integrability conditions for the DMZ system (\ref{equacion de ser carta conjugada}).
\end{remark}

As proved in \cite{KT} we have the following.
\begin{prop}\label{carta conjugada esfera=carta conjugada euclideano}
    Suppose that $h:L^{p+1}\rightarrow\mathbb{S}^n$ has a conjugate chart and $\gamma\in C^{\infty}(L^{p+1})$ non-zero solution of the associated DMZ system, that is $Q(\gamma)=0$. Then the submanifold $H:L^{p+1}\rightarrow\mathbb{R}^{n+1}$ given by $H:=\frac{h}{\gamma}$ satisfy
    \begin{equation}\label{Q(H)=0}
        \tilde{Q}_{ij}(H)=\partial^2_{ij}H-\tilde{\Gamma}^i_{ji}\partial_iH-\tilde{\Gamma}^j_{ij}\partial_jH=0,\quad\forall i<j,
    \end{equation}
    for $\tilde{\Gamma}^i_{ji}=\Gamma^i_{ji}-\frac{\partial_j\gamma}{\gamma}$.
    
    Conversely, let $0\neq H:L^{p+1}\rightarrow\mathbb{R}^{n+1}$ be a submanifold satisfying (\ref{Q(H)=0}). Define $\gamma:=\frac{1}{\|H\|}\neq 0$ and assume that $h:=\gamma H:L^{p+1}\rightarrow\mathbb{S}^{n}$ is an immersion. Then $h$ solves (\ref{equacion de ser carta conjugada}) for $\Gamma_{ji}^i=\tilde{\Gamma}_{ji}^i+\frac{\partial_j\gamma}{\gamma}$ and $g_{ij}=\frac{2\partial_i\gamma\partial_j\gamma-\gamma\tilde{Q}_{ij}(\gamma)}{\gamma^2}$. In this case, $Q(\gamma)=0$.
\end{prop}

This shows that finding conjugate charts for submanifolds in the sphere is equivalent to the problem in the Euclidean space, that is, finding independent solutions to DMZ systems.

\section{Description of the genuine deformations}\label{seccion descripcion del problema}
Our purpose in this section is to find an intermediate analytical characterization for the genuine deformations of a hypersurface $f:M^n\rightarrow\mathbb{R}^{n+1}$ with rank $(p+1)\geq 2$ in higher codimensions. 

\text{ }

From now on, $A=A_\rho$ will denote the shape operator of $f$ with respect to a fixed unit normal vector field $\rho$, $\alpha:=\alpha^g$ the second fundamental form of another isometric immersion $g$ of $M^n$, and $\beta=\alpha\oplus\alpha^f:TM\times TM\rightarrow T^\perp _gM\oplus T^\perp _fM$ the associated flat bilinear form. All sub-indices in this section will be in the range $\{0,1,\ldots,p\}$.

\begin{prop}\label{S(beta) definido si no es reglada}
    Suppose that $f:M^n\rightarrow\mathbb{R}^{n+1}$ has rank $(p+1)$ and fix $\varepsilon\in\{0,1\}$. Let $g:M^n\rightarrow\mathbb{R}^{n+p+\varepsilon}$ be a genuine deformation of $f$ with $p+1+\varepsilon<n$. For $p\geq 5-2\varepsilon$, assume in addition that $f$ and $g$ are not mutually $ (n-p-\varepsilon+2)$-ruled. Then $\mathcal{S}(\beta)$ is non-degenerate on a open dense subset of $M^n$.
\end{prop}

\begin{proof}
    First observe that the condition of not being mutually $(n-p-\varepsilon+2)$-ruled is trivially satisfied for $p\leq 4-2\varepsilon$ by Lemma \ref{AR,R=0}.

    Suppose that there is an open subset $U\subseteq M$ where $\mathcal{S}(\beta)$ is degenerate. Since $W^{p+\varepsilon,1}=T^\perp_gM\oplus T^\perp_fM$ is Lorentzian, there is a smooth unitary normal section $\xi\in T^\perp_gU$ such that
    \begin{equation}\label{Sbeta degenerado}
        \text{span}\{(\xi,\rho)\}=\mathcal{S}(\beta)\cap\mathcal{S}(\beta)^\perp. 
    \end{equation}
    Consider $\gamma:TU\times TU\rightarrow E$ the orthogonal projection of $\alpha^g$ onto $E=\{\xi\}^\perp\subseteq T^\perp_g M$. By (\ref{Sbeta degenerado}), $\gamma$ is flat. Theorems~11 and 14 of \cite{DF3} imply that $f$ and $g$ are simultaneously $R^d$-ruled, where $R^d=\mathcal{N}(\alpha^g_{L^\perp})\cap\mathcal{N}(\alpha^f_{\hat{L}^\perp})$, $L\subseteq\text{span}\langle\xi\rangle$, $\hat{L}\subseteq\text{span}\langle\rho\rangle$, $0\leq\ell=\dim(L)=\dim(\hat{L})\leq 1$ and 
    \begin{equation}\label{cota para el d}
        d\geq n-p-\varepsilon-1+3\ell.
    \end{equation}
    As $f$ and $g$ are not simultaneously $(n-p-\varepsilon+2)$-ruled we have that $L=\hat{L}=\{0\}$ and $R=\Delta_\beta$. By the construction of $L$ in Theorem 11 of \cite{DF3}, this happens only when either $\Delta_\gamma=\Delta_\beta$ or if there is $Z_0\in\Delta_\gamma$ such that $\nabla_{Z_0}^\perp\xi\neq 0$. If $\Delta_\gamma=\Delta_\beta$, by the Main Lemma for $\gamma$ we have that
    $$n-p-\varepsilon+1\leq\dim(\Delta_\gamma)=\dim(\Delta_\beta)\leq \nu_f=n-p-1,$$
    a contradiction. Hence, assume the existence of such $Z_0\in\Delta_\gamma$. 
    
    Call $\phi:TU\times(TU\oplus\text{span}\{\xi\})\rightarrow E$ the map given by
    \begin{equation*}
        \phi(X,v)=(\tilde{\nabla}_X v)_E,
    \end{equation*}
    where $\tilde{\nabla}$ denotes the connection of $\mathbb{R}^{n+p+\varepsilon}$ and the sub-index $E$ denotes the orthogonal projection onto $E$. An easy computation shows that $\phi$ is flat and satisfies Codazzi equation. By the above $\Delta_\phi\subsetneq\Delta_\gamma$. Take $W\in \Delta_\phi$ and $Y\in TU$. Codazzi equation $(\nabla_{Z_0}^E\phi)(W,Y)=(\nabla_W^E\phi)(Z_0,Y)$ reduces to
    \begin{equation*}\label{codazi de phi, caso degenerado}
        \phi([Z_0,W],Y)=\langle AW,Y\rangle\nabla_{Z_0}^{\perp} \xi.
    \end{equation*}
    Using the flatness of $\phi$ and the above relation we get
    $$\langle AW,Y\rangle\| \nabla_{Z_0}^\perp\xi\|^2=\langle\phi([Z_0,W],Y),\phi(Z_0,\xi)\rangle=\langle\phi(Z_0,Y),\phi([Z_0,W],\xi)\rangle=0.$$
    This proves that $\langle AW,Y\rangle=0$ for all $Y\in TU$, since $\nabla^\perp_{Z_0}\xi\neq 0$. Then, $\Delta_\phi\subseteq\Delta_f$, and by Lemma \ref{nulidad para no simetrica}, we have that $\nu_f\geq \dim(\Delta_\phi)\geq n-p-\varepsilon+1$, which is also a contradiction.
\end{proof}
\begin{remark}\label{p=5,6}
    For $p\in\{5-2\varepsilon,6-2\varepsilon\}$ we can prove a weaker version of Proposition \ref{S(beta) definido si no es reglada} without the hypotheses of not being $(n-p+2)$-ruled. In this case we can conclude that either $\mathcal{S}(\beta)$ is non-degenerate, or $f$ and $g$ are mutually $R^d$ ruled with $d=n-p-\varepsilon+2$ and $\Delta_g=\Gamma\subseteq R^d$. Indeed, if we follow the steps of the proof we see that the only problem is when $l=1$. In this case, if $\dim(\Gamma+R^d)\geq n-p-\varepsilon+3$ using Lemma \ref{AR,R=0} for $(\Gamma+R^d)$ we get a contradiction. Then, using (\ref{cota para el d}) we get that $\Delta_\beta=\Gamma\subsetneq R^d$ and $d=n-p-\varepsilon+2$. Finally, just notice that $\Gamma=\Delta_\beta\subseteq\Delta_g\subseteq\Gamma$.
\end{remark}
The Main Lemma gives us the next corollary.
\begin{cor}\label{corolario de no ser reglada}
    If $f$ and $g$ are as in Proposition \ref{S(beta) definido si no es reglada} with $\varepsilon=0$, then $\Delta_g=\Delta_f=\Gamma$ and $\mathcal{S}(\beta)=W^{p,1}$.
\end{cor}

For our purposes, it is more natural and fruitful to classify the deformations in semi-Euclidean spaces, that is, $\mathbb{R}^{n+p}$ with a non-degenerate inner product, which satisfy the same formal properties as the ones in the Euclidean case. In this case we denote the ambient space as $\mathbb{R}^{n+p}_{\mu}$, where $\mu$ is the index of the inner product. In particular, $\mathbb{R}^{n+p}=\mathbb{R}^{n+p}_0$.

\begin{defn}\label{non-degenerate deformation}
    \normalfont Consider $f:M^n\rightarrow\mathbb{R}^{n+q}_{\eta}$ and $g:M^n\rightarrow\mathbb{R}^{n+p}_{\mu}$ two isometric immersions of a Riemannian manifold $M^n$. We say that $g$ is a {\it non-degenerate deformation} of $f$ if there exists $X\in\text{Re}(\beta)$ such that $\beta^X(TM)\subseteq W=T^\perp _gM\oplus T^\perp _fM$ is a non-degenerate subspace, where $\beta=\alpha^g\oplus\alpha^f$. 
\end{defn}

Corollary \ref{corolario de no ser reglada} and Corollary 2 of \cite{Moo} imply the following.
\begin{cor}\label{genuinas son no degeneradas}
    Let $f:M^n\rightarrow\mathbb{R}^{n+1}$ be a rank $p+1<n$ hypersurface. If $p\geq 5$ assume further that $f$ is not $(n-p+2)$-ruled. Then any genuine deformation $g:M^n\rightarrow\mathbb{R}^{n+p}$ of $f$ is non-degenerate. 
\end{cor}

\begin{remark}\label{remark de non degenerate}
    By Lemma \ref{nulidad para no simetrica}, for any non-degenerate deformation $g:M^n\rightarrow\mathbb{R}^{n+p}_{\mu}$ of a nowhere flat hypersurface $f:M^n\rightarrow\mathbb{R}^{n+1}$ of rank $(p+1)$ we have that $\mathcal{S}(\beta)=W$ and $\Delta_g=\Gamma$, as in Corollary \ref{corolario de no ser reglada}.
\end{remark}

The splitting tensor is important in the Sbrana-Cartan classification to differentiate the families of deformable hypersurfaces of rank 2. We will use it in an analogous way.
\begin{defn}\label{definicion spliting tensor}
    \normalfont Consider $M^n$ a Riemannian manifold. For $T\in\Gamma$ we define the {\it splitting tensor with respect to} $T$ as the endomorphism $C_T:\Gamma^\perp\rightarrow\Gamma^\perp$ given by
    \begin{equation*}\label{definicion del splitting tensor}
        C_TX=-(\nabla_X T)^h,
    \end{equation*}
    where $h$ denotes the orthogonal projection on $\Gamma^\perp$. 
\end{defn}

For a non-degenerate deformation $g:M^n\rightarrow\mathbb{R}^{n+p}_{\mu}$ of $f$ (for some $0\leq \mu\leq p$), Remark \ref{remark de non degenerate} and Codazzi equation imply that 
\begin{equation}\label{beta y CT}
    \beta(C_SX,Y)=\beta(X,C_S Y),\quad\forall S\in\Gamma,\quad\forall X,Y\in\Gamma^\perp.
\end{equation}
We introduce the following definition to discard the ruled and surface-like type of situations.

\begin{defn}\label{f generica}
    \normalfont We call $M^n$ {\it generic} it there exists $T\in\Gamma$ such that $C_T$ is semisimple over $\mathbb{C}$. 
\end{defn}

Throughout this section we assume that $g:M^n\rightarrow\mathbb{R}^{n+p}_{\mu}$ is a non-degenerate deformation of $f$ and that $M^n$ is generic. We will classify all such deformations.

\begin{cor}\label{corolario generica implica S(beta)=todo y nulidad}
    Let $f:M^n\rightarrow\mathbb{R}^{n+1}$ be a generic hypersurface of rank $2\leq p+1< n$ and $g:M^n\rightarrow\mathbb{R}^{n+p}_{\mu}$ a non-degenerate deformation. Then, there exists a unique basis (up to order and scalar multiplication) $\{X_i\}_{i=0}^p\in\Gamma^\perp _{\mathbb{C}}$, such that $C_TX_i=\lambda_i(T)X_i\,\,\,\forall T\in\Gamma$. Moreover, for every non-degenerate deformation $g:M^n\rightarrow\mathbb{R}^{n+p}_{\mu}$ of $f$, we have that $\beta(X_i,X_j)=0$ for $i\neq j$.
\end{cor}

\begin{proof}
    Take $T_0\in\Gamma$ such that the eigenvalues of $C_{T_0}$ are distinct, and $C_{T_0}X_i=\lambda_iX_i$. By (\ref{beta y CT}), $\beta(X_i,X_j)=0$ for $i\neq j$ and again by (\ref{beta y CT}) we get that $C_TX_i=\lambda_i(T)X_i$ for some 1-forms $\lambda_i$ on $\Gamma$. This proves that this frame is intrinsic and unique. Moreover, by (\ref{beta y CT}) this frame must diagonalize $\beta$ for all genuine deformations. 
\end{proof}

If $\{X_i\}_{i=0}^p$ are the diagonalizing directions of $\beta$ as above, then after a re-scaling factor, the frame $\{X_i\}$ projects at $L^{p+1}$ as coordinate vectors. More precisely, there exists a chart $(z_1,\ldots,z_s,x_{2s},\ldots, x_p)\in\mathbb{C}^{s}\times\mathbb{R}^{p+1-2s}$ (where $2s$ is the number of non-real eigenvectors of the splitting tensor) such that for the variables $(u_0,\ldots,u_p)=(z_1,\overline{z_1},\ldots, z_s,\overline{z_s},x_{2s},\ldots,x_{p})$ they satisfy
\begin{equation}\label{Bajar coordenadas}
    \partial_i\circ\pi:=\partial_{u_i}\circ\pi=\pi_{*}X_i.
\end{equation}
For a proof of this fact, see Proposition \ref{demostracion Bajar coordenadas} in the Appendix. This chart will be extensively used throughout this work.
These directions also define a conjugation of indices: we denote by $\overline{i}$ the unique index such that $\overline{X_i}=X_{\overline{i}}$. This conjugation will be used without further mention.
Notice also that this coordinate system is unique (up to order and rescale of variables).

Observe now that the set $\{\beta(X_j,X_j)\}_j$ is pointwise a $\mathbb{C}$-basis of $W_\mathbb{C}$. We extend the metrics and the connections of the tangent and normal bundles to their complexifications by $\mathbb{C}$-bilinearity. Then
\begin{equation*}\label{beta(i,i)neq 0}
    \langle\beta(X_i,X_i),\beta(X_i,X_i)\rangle \neq 0,\quad\forall i.
\end{equation*}
Indeed, if $\langle\beta(X_i,X_i),\beta(X_i,X_i)\rangle=0$ for some $i$, by flatness, $\langle\beta(X_i,X_i),\beta(X_j,X_j)\rangle=0$ for all $j$. Since $\mathcal{S}(\beta)=W$ we obtain that that $\beta(X_i,X_i)=0$, which is a contradiction.
Recalling that $\alpha^f(X_i,X_i)=\langle AX_i,X_i\rangle\rho\neq 0$, set
\begin{equation}\label{definion de tau_i}
    \varphi_i:=\frac{\langle\alpha^f(X_i,X_i),\alpha^f(X_i,X_i)\rangle}{\langle \beta(X_i,X_i),\beta(X_i,X_i)\rangle}\neq 0,
\end{equation}
and 
\begin{equation}\label{eta_i}
    \eta_i:=\frac{\alpha^g(X_i,X_i)}{\langle AX_i,X_i\rangle}\in\Gamma(T^\perp_g M\otimes\mathbb{C}).
\end{equation}
Notice that $\varphi_i$ and $\eta_i$ are independent if we change $X_i$ by $\mu_i X_i$ for any $\mu_i\neq 0$. By the flatness of $\beta$, 
\begin{equation}\label{d_ij definicion}
    d_{ij}:=\langle\eta_i,\eta_j\rangle=1+\frac{\delta_{ij}}{\varphi_i},
\end{equation}
where $\delta_{ij}$ is the Kronecker symbol. Since the $p+1$ vectors $\eta_i$ generate the normal of $g$ which has dimension $p$, the matrix $(D_\varphi)_{ij}=d_{ij}$ must be singular. By Lemma \ref{Lema del determinante de D} this is equivalent to 
\begin{equation}\label{suma=-1}
    \varphi_{*}:=-(\varphi_0+\ldots+\varphi_p+1)=0.
\end{equation}
With this, we can verify that
\begin{equation}\label{suma de varphi_i eta_i=0}
    \varphi_0\eta_0+\ldots+\varphi_p\eta_p=0,
\end{equation}
since $\langle\sum_j\varphi_j\eta_j,\eta_k\rangle=\sum_j\varphi_j(1+\frac{\delta_{jk}}{\varphi_k})=0$ for all $k$.
\begin{defn}\label{definicion de ser singular}
    \normalfont We call a tuple $\varphi=(\varphi_i)_{i=0}^p$ {\it admissible} if $\overline{\varphi_i}=\varphi_{\overline{i}}\neq 0$ for all $i$ and satisfies $\varphi_{*}=0$. In this case we denote by
    $2s$ and $P$ the cardinality of the sets $\{i\in\{0,\ldots,p\}|i\neq\overline{i}\}$ and $\{i\in\{0,\ldots,p\}|i=\overline{i}\,\text{ and }\,\varphi_i>0\}$ respectively. We call $p-(s+P)$ the {\it index} of $\varphi$.
\end{defn}

Thus, the collection of functions $\varphi=(\varphi_i)_{i=0}^p$ defined by (\ref{definion de tau_i}) is admissible. 
Moreover, Proposition \ref{singular signature es el indice de la metrica} of the Appendix shows that the index of $\varphi$ is precisely the index $\mu$ of the metric in the ambient space of $g:M^{n}\rightarrow\mathbb{R}^{n+p}_{\mu}$. 

By Codazzi equation for $\alpha$ and $A$, we have that
\begin{equation}\label{eta es delta paralel}
    \nabla^\perp_T\eta_i=0,\quad\forall T\in\Gamma.
\end{equation}
Indeed,
\begin{align*}
    \nabla_T^\perp\eta_i&=\dfrac{\langle AX_i,X_i\rangle(\alpha([T,X_i],X_i)+\alpha(\nabla_TX_i,X_i))-(\langle A[T,X_i],X_i\rangle+\langle A\nabla_TX_i,X_i\rangle)\alpha(X_i,X_i)}{\langle AX_i,X_i\rangle^2}\\
    &=\dfrac{\langle AX_i,X_i\rangle(\langle A[T,X_i],X_i\rangle+\langle A\nabla_TX_i,X_i\rangle)\eta_i-(\langle A[T,X_i],X_i\rangle+\langle A\nabla_TX_i,X_i\rangle)\langle AX_i,X_i\rangle\eta_i}{\langle AX_i,X_i\rangle^2}=0.
\end{align*}
As a consequence of (\ref{eta es delta paralel}) and (\ref{d_ij definicion}), $T(\varphi_i)=0$ for all $i$ and $T\in\Gamma$.

For each $\eta\in (T^\perp _g M)_{\mathbb{C}}$ we define
\begin{equation}\label{definicion de D_eta}
    D_{\eta}=A^{-1}A_{\eta}:\Gamma_{\mathbb{C}}^\perp\rightarrow\Gamma_{\mathbb{C}}^\perp,
\end{equation}
where $A$ is the second fundamental form of $f$ restricted to $\Gamma_{\mathbb{C}}^\perp$ and $A_{\eta}$ is the shape operator of $g$ in the $\eta$ direction also restricted to $\Gamma_{\mathbb{C}}^\perp$.
Since $0=\langle A_\eta X_i,X_j\rangle=\langle AD_\eta X_i,X_j\rangle$ for $i\neq j$, $D_\eta$ is diagonalizable with the same basis $\{X_i\}$. In particular, for $D_i:=D_{\eta_i}$ the Gauss equation implies that
$$D_iX_j=d_{ij}X_j,$$ 
where $d_{ij}$ is defined in (\ref{d_ij definicion}). 

As shown in Lemma 15 of \cite{DFT} we have
\begin{equation}\label{condion para que los Di bajen}
    \nabla_TD_i=[D_i,C_T]=0\quad\forall T\in\Gamma\quad \forall i.
\end{equation}
This motivates the following definition.
\begin{defn}\label{definicion D-system}
    \normalfont Consider a Riemannian manifold $M^n$ of rank $(p+1)\geq 2$. We call a set of smooth tensors $D_i:\Gamma^\perp_{\mathbb{C}}\rightarrow\Gamma^\perp_{\mathbb{C}}$, $i=0,\ldots,p$, a $D${\it -system} if there is a conjugation of indices such that $\overline{D_i}=D_{\overline{i}}$ and the following conditions are satisfied:
    \begin{enumerate}[$i)$]
        \item $\dim_{\mathbb{C}}\ker(D_i-I)=p$, where $I$ is the identity. We denote by $(\frac{1}{\varphi_i}+1)\neq 1$ the remaining eigenvalue of $D_i$ and $X_i$ an associated eigenvector;
        \item $X_j\in\ker(D_i-I)$ for all $j\neq i$;
        \item $\nabla_TD_i=[D_i,C_T]=0\quad\forall T\in\Gamma\quad\forall i$.\label{asdsa}
    \end{enumerate}
\end{defn}
\begin{remark}
    Whenever convenient, we will consider $D_i:(TM)_\mathbb{C}\rightarrow (TM)_\mathbb{C}$ by extending it as zero on $\Gamma_{\mathbb{C}}$. 
\end{remark}
\begin{remark}
    There may be several $D$-systems on $M^n$, but if $M^n$ is generic, then the directions are uniquely determined since the $X_i$'s must also be eigenvectors of the splitting tensor by condition \ref{asdsa}. However, we still have some freedom on the $\varphi_i$'s which determine the $D$-system.
\end{remark}

Let $\phi_{ij}$ be the associated normal connection 1-forms
\begin{equation}\label{definicio de formas de coneccion normal}
    \phi_{ij}(X)=\langle\nabla^\perp_X \eta_i,\eta_j\rangle.
\end{equation}
Clearly $\phi_{ii}=\frac{1}{2}d(\frac{1}{\varphi_i})$ and $\phi_{ij}=-\phi_{ji}$ for $i\neq j$. We denote by $\phi=(\phi_{ij})$ the matrix of 1-forms whose components are $\phi_{ij}$. We can express the normal connection as
\begin{equation}\label{derivada normal de etai}
    \nabla_X^\perp\eta_i=\sum_{j}\phi_{ij}(X)\varphi_j\eta_j.
\end{equation}
Indeed, this is a consequence of (\ref{suma de varphi_i eta_i=0}) and
$$\Big\langle\sum_{j}\phi_{ij}(X)\varphi_j\eta_j,\varphi_k\Big\rangle=\varphi_{ik}(X)+\Big\langle\nabla^\perp _X\eta_i,\sum_j\varphi_j\eta_j\Big\rangle=\phi_{ik}(X),\quad\forall k.$$

The next result gives a bijection between the set of non-degenerate deformations of $f$ in codimension $p$ and the set of pairs $(D,\phi)$ satisfying certain equations. 
\begin{prop}\label{caracterizacion en M dim p+q+1}
    Consider a simply connected generic hypersurface $f:M^n\rightarrow\mathbb{R}^{n+1}$ of rank 
    $2\leq p+1<n$. Let $g:M^n\rightarrow\mathbb{R}^{n+p}_{\mu}$ be a non-degenerate deformation of $f$ (for some $0\leq \mu\leq p$). Then there exist a $D$-system and a $(p+1)\times(p+1)$ matrix of 1-forms $\phi=(\phi_{ij})$ satisfying: 
    \begin{enumerate}[a)]
        \item $\varphi$ is admissible of index $\mu$;\label{singular y def positivo en M}
        \item $\overline{\phi_{ij}(X)}=\phi_{\overline{i}\,\overline{j}}(\overline{X})$\label{varphi y phi admisible};
        \item $AD_i=D_i^t A$; \label{A_i es simetrico caso p+q+1 en M}
        \item $\sum_k\varphi_k\phi_{ik}=0$, $\forall i$; \label{kernel de S esta en el kernel de phi}
        \item $\phi_{ij}+\phi_{ji}=0$ for $i\neq j$ and $\phi_{ii}=\frac{1}{2}d(\frac{1}{\varphi_i})$; \label{phi_ij+phi_ji caso p+q+1 en M} 
        \item $\phi_{ij}(T)=d\phi_{ij}(Z,T)=0$ for any $Z$ and $T\in\Gamma$;\label{phi_ij baja}
        \item $\nabla_X(AD_i)Y-\nabla_Y(AD_i)X=A\Big(\sum_{j}\varphi_j(\phi_{ij}\wedge D_j)(X,Y)\Big),\quad\forall i,X,Y\in TM$; \label{Codazzi caso p+q+1 en M}
        \item $\langle [AD_i,AD_j]X,Y\rangle=d\phi_{ij}(X,Y)+\Omega_{ij}(X,Y)$, $\forall i,j$ and $X,Y\in TM$, where $\Omega=(\Omega_{ij})$ is the matrix of 2-forms given by 
        $\Omega_{ij}=\sum_k\varphi_k(\phi_{ik}\wedge\phi_{jk})$.\label{Ricci caso p+q+1 en M}
    \end{enumerate}
    Conversely, suppose that we have a $D$-system and a $(p+1)\times(p+1)$ matrix of $1$-forms $\phi=(\phi_{ij})$ satisfying the conditions \ref{singular y def positivo en M} to \ref{Ricci caso p+q+1 en M} above. Then, there exists an isometric immersion $g=g_{(D,\phi)}:M^n\rightarrow\mathbb{R}^{n+p}_{\mu}$ which is a genuine deformation of $f$ determined by $D$ and $\phi$.
    Moreover, given two pairs $(D,\phi),(\hat{D},\hat{\phi})$ that satisfy the above properties, then $g_{(D,\phi)}$ and $\hat{g}_{(\hat{D},\hat{\phi})}$ are congruent if and only if $(D,\phi)=(\hat{D},\hat{\phi})$.
\end{prop}
\begin{proof}
    We have already proved that if $g:M^n\rightarrow\mathbb{R}^{n+p}_{\mu}$ is a deformation for $f$, then there is such a pair $(D,\phi)$ satisfying all the above properties. Indeed, observe that $AD_i=A_{\eta_i}$ is a symmetric tensor, \ref{Codazzi caso p+q+1 en M} is Codazzi equation for $A_{\eta_i}$, and \ref{Ricci caso p+q+1 en M} is just Ricci equation expressed as
    $$\langle R^\perp(X,Y)\eta_i,\eta_j\rangle=X\langle\nabla_Y^\perp \eta_i,\eta_j\rangle-Y\langle\nabla_X^\perp \eta_i,\eta_j\rangle-\langle\nabla_{[X,Y]}^\perp \eta_i,\eta_j\rangle+\langle\nabla_X^\perp\eta_i,\nabla_Y^\perp\eta_j\rangle-\langle\nabla_Y^\perp\eta_i,\nabla_X^\perp\eta_j\rangle.$$
    Moreover, if $g:M^n\rightarrow\mathbb{R}^{n+p}_{\mu}$ and $\hat{g}:M^n\rightarrow\mathbb{R}^{n+p}_{\mu}$ are two isometric immersion with the same associated pair $(D,\phi)$, then they are congruent. 
    Indeed, define $t:(T^\perp _gM)_{\mathbb{C}}\rightarrow (T^\perp _{\hat{g}}M)_{\mathbb{C}}$ by $t(\eta_i)=\hat{\eta}_i$, where the $\eta_i$'s are defined by (\ref{eta_i}), and similarly for the $\hat{\eta}_i$'s. It is easy to verify that $t$ is a well defined parallel bundle isometry which preserves the respective second fundamental form, $t\circ\alpha^g=\alpha^{\hat{g}}$. By the Fundamental Theorem of submanifolds this map induces an isometry $T:\mathbb{R}^{n+p}_{\mu}\rightarrow\mathbb{R}^{n+p}_{\mu}$ such that $\hat{g}=T\circ g$.
    
    Let us prove the converse. The main idea is to consider the bundle $E=\mathbb{C}^{p+1}/\ker(D_\varphi)\rightarrow M^n$ as a candidate to be the complexification of the normal bundle for $g$ and use the pair $(D,\phi)$ to define a second fundamental form, a metric and a connection on $E$. Then, the Fundamental Theorem of submanifolds will imply the existence of $g$. We denote the elements of $E\rightarrow M$ with brackets to differentiate them from those of $\mathbb{C}^{p+1}$. 
    
    Consider on $E$ the bilinear product defined by $\langle [e_i],[e_j]\rangle= d_{ij}=1+\frac{\delta_{ij}}{\varphi_i}$. By Proposition \ref{singular signature es el indice de la metrica} of the Appendix, this defines a non-degenerate inner product on the real bundle $\text{Re}_{C}(E)\rightarrow M^n$ of index $\mu$, where the conjugation is given on the canonical basis by $C([e_i])=[e_{\overline{i}}]$.
    
    Equation (\ref{derivada normal de etai}) induces the connection $\tilde{\nabla}_X e_i=\sum_j\varphi_j\phi_{ij}(X)e_j$ on the trivial bundle $\mathbb{C}^{p+1}\rightarrow M$. This connection descends to the quotient $E$. Indeed, using \ref{kernel de S esta en el kernel de phi},  \ref{phi_ij+phi_ji caso p+q+1 en M} and (\ref{kernel de S}) we get
    \begin{equation*}
        \tilde{\nabla}_X\Big(\sum_j\varphi_j e_j\Big)=\sum_k\Big(X(\varphi_k)+\varphi_k\big(\sum_j\varphi_j\phi_{jk}(X)\big)\Big)e_k=\sum_k\Big(X(\varphi_k)+\varphi_k(2\varphi_k\phi_{kk}(X))\Big)e_k=0.
    \end{equation*}
    Thus, $\nabla^E_X[e_i]=\sum_j \varphi_j\phi_{ij}(X)[e_j]$ is a well defined connection on $E\rightarrow M^n$. By \ref{phi_ij+phi_ji caso p+q+1 en M}, this connection is compatible with the product induced by $D_\varphi$. Indeed, notice that 
    \begin{align*}
        \langle\nabla^E_X[e_i],[e_j]\rangle=\sum_k \varphi_k\phi_{ik}(X)d_{kj}=\phi_{ij}(X)+\sum_k \varphi_k\phi_{ik}(X)=\phi_{ij}(X),
    \end{align*}
    and then $\langle\nabla^E_X[e_i],[e_j]\rangle+\langle[e_i],\nabla^E_X[e_j]\rangle=\phi_{ij}(X)+\phi_{ji}(X)=X(d_{ij})=X\langle[e_i],[e_j]\rangle$.
    
    For $X,Y\in (T_xM)_{\mathbb{C}}$ we define the linear map $\ell_{X,Y}:\mathbb{C}^{p+1}\rightarrow\mathbb{C}$ by $\ell_{X,Y}(e_i)=\langle AD_iX,Y\rangle$. Then, by (\ref{kernel de S}),
    \begin{equation*}
        \ell_{X,Y}\Big(\sum_j\varphi_je_j\Big)=\Big\langle A\Big(\sum_j D_j\varphi_j\Big)X,Y\Big\rangle=0.
    \end{equation*}
    Thus there exists a unique $\gamma(X,Y)\in E$ such that $\langle\gamma(X,Y),[e_i]\rangle=\langle AD_iX,Y\rangle$ for all $i$. This tensor $\gamma$ is symmetric by \ref{A_i es simetrico caso p+q+1 en M} and by definition $\Gamma\subseteq\Delta_\gamma$. Observe that
    \begin{equation}\label{gamma(X_i,X_i) diagonaliza}
        \gamma(X_i,X_i)=\langle AX_i,X_i\rangle [e_i]\quad\forall i,
    \end{equation}
    \begin{equation}\label{gamma(i,j)=0}
        \gamma(X_i,X_j)=0 \quad\forall i\neq j,
    \end{equation}
    since 
    $$\langle\langle AX_i,X_i\rangle [e_i],[e_k]\rangle=\langle AX_i,X_i\rangle d_{ik}=\langle\gamma(X_i,X_i),[e_k]\rangle\quad\forall k,$$  $$\langle\gamma(X_i,X_j),[e_k]\rangle=\langle AD_kX_i,X_k\rangle=d_{ki}\langle AX_i,X_j\rangle=0\quad\forall k.$$
    Equations (\ref{gamma(X_i,X_i) diagonaliza}) and (\ref{gamma(i,j)=0}) show that $\Delta_\gamma=\Gamma$, $\{X_i\}_{i=0}^p$ diagonalizes $\gamma$, and $\mathcal{S}(\beta)=E\oplus T^\perp _fM$ where $\beta=\gamma\oplus\alpha^f$.
    Notice that
    $$\langle\gamma(X_i,X_i),\gamma(X_j,X_j)\rangle=\langle AX_i,X_i\rangle\langle AX_j,X_j\rangle d_{ij}=\langle AX_i,X_i\rangle\langle AX_j,X_j\rangle,\quad\forall i\neq j.$$
    This proves that $\gamma$ satisfies Gauss equation on $(TM)_{\mathbb{C}}$ since all the other Gauss equations are trivially satisfied since $\{X_0,\ldots,X_p\}$ is a basis of $\Gamma^\perp_\mathbb{C}$ which simultaneously diagonalizes $\gamma$ and $\alpha^f$. 
    
    To verify that $\gamma$ is a Codazzi tensor, just observe that, for all $X,Y,Z$, we have
    \begin{align*}
        \langle(\nabla^E_X\gamma)(Y,Z),[e_i]\rangle&=X(\langle\gamma(Y,Z),[e_i]\rangle)-\langle\gamma(\nabla_X Y,Z),[e_i]\rangle-\langle\gamma(Y,\nabla_X Z),[e_i]\rangle-\langle\gamma(Y,Z),\nabla^E_X[e_i]\rangle\\
        &=X(\langle AD_iY,Z\rangle)-\langle AD_i\nabla_X Y,Z\rangle-\langle AD_iY,\nabla_XZ\rangle-\sum_{j}\varphi_j\phi_{ij}(X)\langle AD_j Y,Z\rangle\\
        &=\langle\nabla_X(AD_i)Y,Z\rangle-\sum_{j}\varphi_j\phi_{ij}(X)\langle AD_j Y,Z\rangle.
    \end{align*}
    This expression is symmetric for $X,Y$ by \ref{Codazzi caso p+q+1 en M}.
    
    Lastly, Ricci equation follows from
    \begin{align*}
        \langle R(X,Y)[e_i],[e_j]\rangle&=X(\langle\nabla^E_Y[e_i],[e_j]\rangle)-Y(\langle\nabla^E_X[e_i],[e_j]\rangle)-\langle\nabla^E_{[X,Y]}[e_i],[e_j]\rangle\\
        &\quad-\langle\nabla^E_Y[e_i],\nabla^E_X[e_j]\rangle+\langle\nabla^E_X[e_i],\nabla^E_Y[e_j]\rangle\\
        =&\, d\phi_{ij}(X,Y)+\Big\langle\sum_k \nabla^E_X[e_i],\sum_k \varphi_k\phi_{jk}(Y)[e_k]\Big\rangle-\Big\langle\nabla^E_Y[e_i],\sum_k \varphi_k\phi_{jk}(X)[e_k]\Big\rangle\\
        =&\, d\phi_{ij}(X,Y)+\Omega_{ij}(X,Y)=\langle[AD_i,AD_j]X,Y\rangle.
    \end{align*}
    
    We conclude from the Fundamental Theorem of submanifolds that there exists an isometric immersion $g=g_{(D,\phi)}:M^n\rightarrow\mathbb{R}^{n+p}_{\mu}$ such that the complexification of the normal bundle is $(E,\nabla^E)$ and the second fundamental form of $g$ is $\gamma$, up to a parallel isometry of vector bundles. Moreover, $g$ is a non-degenerate deformation, since $X=\sum_i X_i\in TM$ verifies that $\beta^X:\Gamma^\perp\rightarrow W$ is an isomorphism.
\end{proof}
\subsection{Projecting to the nullity leaf space}
Since we now have a description of the genuine deformations in terms of pairs $(D,\phi)$, we proceed to reduce the problem to the nullity leaf space $L^{p+1}=M^n/\Gamma$, and characterize each condition of Proposition \ref{caracterizacion en M dim p+q+1} in terms of $\varphi$ and the Gauss parametrization data $(h,\gamma)$ of the hypersurface $f$.
\\\text{ }

First, we translate Proposition \ref{caracterizacion en M dim p+q+1} to the leaf space, which is a crucial point in our argument. We denote by $\langle\cdot,\cdot\rangle'$ and $\nabla'$ the metric and the connection induced by the Gauss map $h:L^{p+1}\rightarrow\mathbb{S}^n$. 

\begin{prop}\label{Caracterizacion en Lp+q+1}
    Let $f:M^n\rightarrow\mathbb{R}^{n+1}$ be a rank $(p+1)$ hypersurface. Consider the nullity leaf space $\pi:M^n\rightarrow L^{p+1}=M^n/\Gamma$, and $\gamma\in C^\infty(L^{p+1})$, $h:L^{p+1}\rightarrow\mathbb{S}^n$ the Gauss parametrization data of $f$. If $(D,\phi)$ is a pair on $M^n$ as in Proposition \ref{caracterizacion en M dim p+q+1}, then there is an induced pair $(\hat{D},\hat{\phi})$ on $L^{p+1}$ such that 
    $$\hat{\varphi}_i\circ\pi=\varphi_i,\quad\hat{D}_i\circ\pi_{*}=\pi_{*}\circ D_i,\quad\hat{\phi}_{ij}\circ\pi_{*}=\phi_{ij}.$$
    In addition, $(\hat{D},\hat{\phi})$ satisfies for $\pi_{*}X=\hat{X},\pi_{*}Y=\hat{Y}\in TL$:
    \begin{enumerate}[i)]
        \item $\hat{\varphi}$ is admissible of index $\mu$;\label{dsucc 0 en Lp}
        \item $\overline{\hat{\phi}_{ij}(X)}=\hat{\phi}_{\overline{i}\,\overline{j}}(\overline{\hat{X}})$;\label{varphi y phi admisibles en L}
        \item $(\text{Hess}_\gamma+\gamma I)\hat{D}_i=\hat{D}_i^t(\text{Hess}_\gamma+\gamma I)$;\label{hess gamma D simetrico caso p+q+1}
        \item $\alpha^h(\hat{D}_i \hat{X},\hat{Y})=\alpha^h(\hat{X},\hat{D}_i \hat{Y})$;\label{alpha Di p+q+1}
        \item $\sum_k\hat{\varphi}_{k}\hat{\phi}_{ik}=0,\quad \forall i$; \label{kernel de S esta en el kernel de phi en Lp+1}
        \item $\hat{\phi}_{ij}+\hat{\phi}_{ji}=0$ for $i\neq j$ and $\hat{\phi}_{ii}=\frac{1}{2}d(\frac{1}{\hat{\varphi}_i})$; \label{compatibilidad de conexion en Lp+q+1}
        \item $(\nabla'_{\hat{X}}\hat{D}_i)\hat{Y}-(\nabla'_{\hat{Y}}\hat{D}_i)\hat{X}=\sum_{j}\hat{\varphi}_j(\hat{\phi}_{ij}\wedge\hat{D}_j)(\hat{X},\hat{Y}),\quad\forall i$;\label{codazzi en Lp+q+1}
        
        \item $\langle \hat{D}_j\hat{X},\hat{D}_i \hat{Y}\rangle'-\langle \hat{D}_i \hat{X},\hat{D}_j \hat{Y}\rangle'=d\hat{\phi}_{ij}(\hat{X},\hat{Y})+\hat{\Omega}_{ij}(\hat{X},\hat{Y})$ where $\hat{\Omega}_{ij}\circ\pi_{*}=\Omega_{ij}$. \label{Ricci en L p+q+1}
    \end{enumerate}
    Conversely, if $(h,\gamma)$ and $(\hat{D},\hat{\phi})$ satisfy \ref{dsucc 0 en Lp}-\ref{Ricci en L p+q+1} above, then they give rise, via the Gauss parametrization, to a hypersurface $f$ and a pair $(D,\phi)$ satisfying Proposition \ref{caracterizacion en M dim p+q+1}.
\end{prop}
\begin{proof}
    From Corollary 12 of \cite{DFT}, we know that $D_i$, $\varphi$, $\phi$ and $\Omega$ descend to the quotient by definition of a $D$-system and \ref{phi_ij baja} of Proposition \ref{caracterizacion en M dim p+q+1}. 
    
    Let $\rho$ be the Gauss map of $f$. Then $f_{*}AX=-\rho_{*}X=-h_{*}\pi_{*}X$. Take $X,Y$ projectable vector fields on $M^n$, $\hat{X}\circ\pi=\pi_{*}X$, $\hat{Y}\circ\pi=\pi_{*}Y$. We see that \ref{Ricci en L p+q+1} comes from \ref{A_i es simetrico caso p+q+1 en M} and \ref{Ricci caso p+q+1 en M} of Proposition \ref{caracterizacion en M dim p+q+1} since
    $$\langle AD_jX, AD_iY\rangle-\langle AD_iX,AD_jY\rangle=\langle \hat{D}_j\pi_{*}X,\hat{D}_i\pi_{*}Y\rangle'-\langle\hat{D}_i\pi_{*}X,\hat{D}_j\pi_{*}Y\rangle'.$$
    Notice that
    \begin{align*}
        f_{*}\nabla_X AD_iY&=\widetilde{\nabla}_X f_{*}AD_iY -\langle AX,AD_iY\rangle \rho=-\widetilde{\nabla}_X h_{*}\pi_{*}D_i Y-\langle h_{*}\pi_{*}X,h_{*}\pi_{*} D_iY\rangle h\circ\pi\\
        &=-h_{*}\nabla'_{\hat{X}}\hat{D}_i\hat{Y}-\alpha^h(\hat{X},\hat{D}_i\hat{Y}).
    \end{align*}
    Hence, using this in \ref{Codazzi caso p+q+1 en M} of Proposition \ref{caracterizacion en M dim p+q+1} we obtain \ref{alpha Di p+q+1} and \ref{codazzi en Lp+q+1}. By the Gauss parametrization $\Phi:U\subseteq T^\perp_h L\rightarrow M$ and $\psi(w)=f\Phi(w)=\gamma h+h_{*}\nabla\gamma+w$, $w\in T^\perp_h L$, we get
    $$\psi_{*}X=h_{*}P\hat{\pi}_{*}X+\alpha^h(\hat{\pi}_{*}X,\nabla'\gamma),$$
    where $\hat{\pi}:T^\perp_h L\rightarrow L$ is the bundle projection, $X\in T_{w}(T^\perp_h L)$ is an horizontal vector, and $P$ is the symmetric tensor 
    \begin{equation}\label{P_w}
        P=P_w=\text{Hess}_\gamma+\gamma I-B_w:TL\rightarrow TL,
    \end{equation}
    where $B_w$ is the shape operator of $h$ in the $w-$direction. This implies that
    \begin{equation*}
        -\langle AD_i\Phi_{*}X,\Phi_{*}Y\rangle=\langle h_{*}\hat{D}_i\pi_{*}\Phi_{*}X,h_{*}P\pi_{*}Y\rangle=\langle \hat{D}_i\hat{\pi}_{*}X,P\hat{\pi}_{*}Y\rangle^{'}.
    \end{equation*}
    Therefore $\hat{D}_i^tP=P\hat{D}_i$ and as $D_i^t B_w=B_w D_i$ by \ref{alpha Di p+q+1}, we conclude \ref{hess gamma D simetrico caso p+q+1}.
    
    The converse follows easily by defining $D_i(\Gamma)=0$  and $\pi_{*}D_i X=\hat{D}_i \pi_{*}X$ for $X\in\Gamma^\perp$.
\end{proof}

From now on, we will drop the hat over variables and the prime for the metric and connection of $h:L^{p+1}\rightarrow\mathbb{S}^n$, since we now focus on the leaf space and not on the manifold $M^n$.

The main idea will be to express Proposition \ref{Caracterizacion en Lp+q+1} in terms of the coordinate system given by (\ref{Bajar coordenadas}). As $(\partial_j:=\partial_{u_j})_j$ is a basis on $(TL)_\mathbb{C}$, all the indices will be with respect to this basis between $0$ and $p$. Notice that, since the coordinate vectors are the eigenvalues of the $D_i$'s, they are completely determined by $\varphi$. 

As was shown in the proof of Proposition \ref{caracterizacion en M dim p+q+1} $\varphi$ is used to define the second fundamental form and the metric of the normal bundle of $g$. On the other hand, $\phi$ is used to define the normal connection. Since Codazzi equation relates the second fundamental form with the normal connection, we expect that the $\phi$ is related with $\varphi$. In fact, $\varphi$ determines $\phi$ completely:
\begin{lema}\label{derivadas de phi, Gammas=0, integrabilidad}
    Let $(D,\phi)$ be a pair as in Proposition \ref{Caracterizacion en Lp+q+1}. Then, \ref{codazzi en Lp+q+1} (Codazzi equation) and \ref{compatibilidad de conexion en Lp+q+1} (compatibility of the connection with the metric) are equivalent to $\phi$ being uniquely determined by $\varphi$ by the followings conditions:
    \begin{equation}\label{phi_is(r) para r distinto de i y distinto de s}
        \phi_{is}(\partial_r)=0 \quad\forall r\neq i\neq s\neq r,
    \end{equation}
    \begin{equation}\label{phi_is(i)}
        \phi_{is}(\partial_i)=-\frac{\Gamma_{is}^s}{\varphi_i},\quad\forall s\neq i,
    \end{equation}
    \begin{equation}\label{phi_is(s)}
        \phi_{is}(\partial_s)=\frac{\Gamma_{si}^i}{\varphi_s},\quad\forall s\neq i,
    \end{equation}
    and
    \begin{equation}\label{Gamma es cero si tres son diferentes en p+1}
        \Gamma^k_{ij}=0 \quad\forall i\neq j\neq k\neq i,
    \end{equation}
    \begin{equation}\label{derivadas de phi_i caso p+1}
        \partial_j\varphi_i=2\Gamma^i_{ji}\varphi_i\quad\forall  i\neq j.
    \end{equation}
\end{lema}
\begin{proof}
Take in \ref{codazzi en Lp+q+1} of Proposition \ref{Caracterizacion en Lp+q+1} $X=\partial_r$, $Y=\partial_s$ with $s\neq r$. Then
    \begin{equation}\label{valor conexion primera}
        \partial_r(d_{is})+(d_{is}-d_{ir})\Gamma_{rs}^s=\sum_j\phi_{ij}(\partial_r)d_{js}\varphi_j=\phi_{is}(\partial_r)(d_{ss}-1)\varphi_s+\sum_j\phi_{ij}(\partial_r)d_j\varphi_j=\phi_{is}(\partial_r),
    \end{equation}
    $$(d_{is}-d_{ir})\Gamma_{rs}^t=0,\quad\forall t\neq r,s,$$
    and symmetric equations interchanging $r$ with $s$. In particular for $i=s$, we get (\ref{Gamma es cero si tres son diferentes en p+1}) and
    \begin{equation*}
        \partial_r(d_{ii})+2\big(d_{ii}-1\big)\Gamma_{ri}^i=0,\quad\forall r\neq i,
    \end{equation*}
    which is an equivalent form of (\ref{derivadas de phi_i caso p+1}). 
    Using (\ref{derivadas de phi_i caso p+1}) in (\ref{valor conexion primera}) we get (\ref{phi_is(r) para r distinto de i y distinto de s}) and (\ref{phi_is(i)}).
    Equation \ref{compatibilidad de conexion en Lp+q+1} of Proposition \ref{Caracterizacion en Lp+q+1}, for $X=\partial_s$ and $j=s$ implies (\ref{phi_is(s)}).
\end{proof}

By \ref{dsucc 0 en Lp} of Proposition \ref{Caracterizacion en Lp+q+1}, we can use (\ref{derivadas de phi_i caso p+1}) to get
    \begin{equation}\label{derivada_i de varphi_i}
        \partial_i\varphi_i=-2\sum_{j\neq i}\Gamma_{ij}^j\varphi_j.
    \end{equation}
This implies the following.
\begin{cor}\label{corolario f}
    The pair $(D,\phi)$ are determinated by an admissible function $\varphi=(\varphi_i)_{i=0}^p$ satisfying
    $$\partial_i\varphi_j=2\Gamma_{ij}^j\varphi_j\text{ for } i\neq j,\,\text{ and }\,\partial_i\varphi_i=-2\sum_{j\neq i}\Gamma_{ij}^j\varphi_j.$$ 
    In particular, the moduli space of genuine deformations of $f$ has finite dimension at most $p$.
\end{cor}

\begin{remark}
    Since $\varphi$ is admissible, the matrix of 1-forms $\phi$ defined by (\ref{phi_is(r) para r distinto de i y distinto de s}), (\ref{phi_is(i)}) and (\ref{phi_is(s)}) immediately satisfy \ref{varphi y phi admisibles en L}, \ref{kernel de S esta en el kernel de phi en Lp+1} and \ref{compatibilidad de conexion en Lp+q+1} of Proposition \ref{Caracterizacion en Lp+q+1}. Indeed, by (\ref{derivadas de phi_i caso p+1})
    \begin{align*}
        \sum_s\phi_{is}(\partial_r)\varphi_s&=\phi_{ii}(\partial_r)\varphi_i+\phi_{ir}(\partial_r)\varphi_r=\frac{1}{2}\partial_r(\varphi^{-1}_i)\varphi_i+\Gamma_{ri}^i=0,
    \end{align*}
    and by (\ref{suma=-1}) we get
    \begin{align*}
        \sum_s\phi_{is}(\partial_i)\varphi_s&=\frac{1}{2}\partial_i(\varphi^{-1}_i)\varphi_i-\sum_{s\neq i}\frac{\Gamma_{is}^s\varphi_s}{\varphi_i}=\frac{1}{2\varphi_i}\partial_i\Big(\sum_s\varphi_s\Big)=0.
    \end{align*}
    From now on, whenever we work with with $\phi$ we will assume that it is defined by $\varphi$ by (\ref{phi_is(r) para r distinto de i y distinto de s}), (\ref{phi_is(i)}) and (\ref{phi_is(s)}).
\end{remark}
\begin{lema}\label{Condiciones en (h,gamma)}
    Condition \ref{alpha Di p+q+1} of Proposition \ref{Caracterizacion en Lp+q+1} is equivalent to
    \begin{equation}\label{alpha diagonaliza en coordenadas en p+q+1}
    \alpha^h(\partial_j,\partial_k)=0\quad\forall j\neq k.
\end{equation}
    In particular, the chart is a conjugate chart. Moreover, condition \ref{hess gamma D simetrico caso p+q+1} of Proposition \ref{Caracterizacion en Lp+q+1} is equivalent to the support function  $\gamma$ satisfying $Q(\gamma)=0$.
\end{lema}
\begin{proof}
    Take $X=\partial_j$ and $Y=\partial_k$ in \ref{alpha Di p+q+1} for $j\neq k$. Then $(d_{ij}-d_{ik})\alpha^h(\partial_j,\partial_k)=0$ for all $i$. We obtain (\ref{alpha diagonaliza en coordenadas en p+q+1}) from this for $i=j$. 
    Using (\ref{Gamma es cero si tres son diferentes en p+1}) and Remark \ref{condicion integrabilidad carta conjugada} we conclude that the chart is conjugate.
    
    The last assertion follows by evaluating the bilinear map given by \ref{hess gamma D simetrico caso p+q+1} of Proposition \ref{Caracterizacion en Lp+q+1} on the coordinates fields $X=\partial_j$ and $Y=\partial_k$ for $j\neq k$. 
\end{proof}
The only remaining condition  to analyze is \ref{Ricci en L p+q+1}, Ricci equation. We see now that, by Remark \ref{condicion integrabilidad carta conjugada} and Lemmas \ref{derivadas de phi, Gammas=0, integrabilidad} and \ref{Condiciones en (h,gamma)}, this is trivially satisfied.
\begin{lema}
    Assume that $(D,\phi)$ satisfy conditions \ref{varphi y phi admisibles en L} to \ref{codazzi en Lp+q+1} of Proposition \ref{Caracterizacion en Lp+q+1}. Then \ref{Ricci en L p+q+1} of Proposition \ref{Caracterizacion en Lp+q+1} is satisfied if and only (\ref{integrabilidad definicion de conjugada}) holds.
\end{lema}
\begin{proof}
    By (\ref{phi_is(r) para r distinto de i y distinto de s}), the only non zero equations of \ref{Ricci en L p+q+1} are when $X=\partial_j, Y=\partial_r$ for $r\neq j$. First, for $r\neq i,j$ we get
    \begin{equation}\label{dphi_ij(j,r)}
        d\phi_{ij}(\partial_j,\partial_r)=\partial_j(\phi_{ij}(\partial_r))-\partial_r(\phi_{ij}(\partial_j))=-\partial_r\Big(\frac{\Gamma_{ji}^i}{\varphi_j}\Big)=\frac{-\partial_r\Gamma_{ji}^i+2\Gamma_{rj}^j\Gamma_{ji}^i}{\varphi_j},
    \end{equation}
    and 
    \begin{align*}
        \Omega_{ij}(\partial_j,\partial_r)&=(\phi_{ij}(\partial_j)\phi_{jj}(\partial_r))\varphi_j+(-\phi_{ir}(\partial_r)\phi_{jr}(\partial_j))\varphi_r+(-\phi_{ii}(\partial_r)\phi_{ji}(\partial_j))\varphi_i\\
        &=\frac{\Gamma_{ji}^i\partial_r(\varphi^{-1}_j)}{2}+\frac{\Gamma_{ri}^i\Gamma_{jr}^r}{\varphi_j}+\frac{\partial_r(\varphi^{-1}_i)\Gamma_{ji}^i\varphi_i}{2\varphi_j}.
    \end{align*}
    Therefore, by (\ref{derivadas de phi_i caso p+1}) we have
    \begin{equation}\label{Omega(j,r)}
        \Omega_{ij}(\partial_j,\partial_r)=\frac{-\Gamma_{ji}^i\Gamma_{rj}^j+\Gamma_{ri}^i\Gamma_{jr}^r-\Gamma_{ri}^i\Gamma_{jr}^r}{\varphi_j}.
    \end{equation}
    Adding (\ref{dphi_ij(j,r)}) and (\ref{Omega(j,r)}) we get (\ref{integrabilidad definicion de conjugada}).
    
    For $X=\partial_j$ and $Y=\partial_i$, first notice that
    \begin{align*}
        \sum_{k}\varphi_k(d\phi_{ik}+\Omega_{ij})&=\sum_k\varphi_kd\phi_{ik}+\sum_l\varphi_l\phi_{il}\wedge\Big(\sum_k\varphi_k\phi_{kl}\Big)=\sum_k\varphi_kd\phi_{ik}+\sum_l\varphi_l\phi_{il}\wedge(2\varphi_l\phi_{ll})\\
        &=\sum_k\varphi_kd\phi_{ik}-\sum_{l}\phi_{il}\wedge d\varphi_l=d\Big(\sum_k\varphi_k\phi_{ik}\Big)=0. 
    \end{align*}
    Then using that $\sum_k\varphi_kD_k=0$ we conclude that
    \begin{align*}
        0&=\sum_{k}\varphi_k(d\phi_{ik}+\Omega_{ik})(\partial_j,\partial_i)=\varphi_j[d\phi_{ij}(\partial_j,\partial_i)+\Omega_{ij}(\partial_j,\partial_i)]+\sum_{k\neq j}\varphi_k[\langle D_k\partial_j,D_i\partial_i\rangle-\langle D_i\partial_j,D_k\partial_i\rangle]\\
        &=\varphi_j[d\phi_{ij}(\partial_j,\partial_i)+\Omega_{ij}(\partial_j,\partial_i)]-\varphi_j[\langle D_j\partial_j,D_i\partial_i\rangle-\langle D_i\partial_j,D_j\partial_i\rangle].
    \end{align*}
    This shows that all Ricci equations are satisfied. 
\end{proof}
\begin{remark}\label{Qij(xik)=0}
    Equation (\ref{integrabilidad definicion de conjugada}) can also be expressed as
    $$Q_{ij}(\xi_k)=0\quad\forall i\neq j\neq k\neq i,$$
    where $\xi_k$ is a (possibly complex) local smooth square root of $\varphi_k$.
\end{remark}
The last results motivate the following definition.
\begin{defn}\label{definicion S y S*}
    \normalfont Given $h:L^{p+1}\rightarrow\mathbb{S}^n$ with a conjugate chart, let
    \begin{equation*}
        \mathcal{S}_\mu^{*}=\{\varphi\text{ is admissible of index }\mu \text{ and } \partial_i\varphi_j=2\Gamma_{ij}^j\varphi_j,\forall i\neq j\},
    \end{equation*}
    \begin{equation*}
        \mathcal{S}^{*}=\bigcup_{\mu=0}^p\mathcal{S}_\mu^{*}=\{\varphi\text{ is admissible and } \partial_i\varphi_j=2\Gamma_{ij}^j\varphi_j,\forall i\neq j\}.
    \end{equation*}    
\end{defn}
\begin{remark}\label{relacion caso p=2 con trabajo dft}
The moduli space $\mathcal{C}_h$ described in Theorem 1 of \cite{DFT} is naturally related to our moduli space $\mathcal{S}^*$. 
Suppose that $H:L^3\rightarrow\mathbb{S}^n$ has a conjugate chart $(u_0,u_1,u_2)$ centered at the origin with $u_2$ real and $\mathcal{S}_{0}^*\neq\emptyset$. Let $L^2=\{u=0\}\subseteq L^3$ and  $h=H|_{L^2}$. Then there is an injection $\mathcal{S}_{0}^*\rightarrow\mathcal{C}_h$ given by
    $$\varphi=(\varphi_0,\varphi_1,\varphi_2)\rightarrow \Big(\frac{1}{2}\varphi_0|_{u_1=u_2=0},\frac{1}{2}\varphi_1|_{u_0=u_2=0}\Big),\quad\text{if } (u_0,u_1) \text{ are real},$$
    $$\varphi=(\varphi_0,\varphi_1,\varphi_2)\rightarrow \frac{1}{2}\varphi_0|_{u_1=u_2=0},\quad\text{if } (u_0,u_1) \text{ are complex}.$$
    Indeed, using the notation in \cite{DFT}, the condition $Q(\rho^{UV})=0$ in the real case is just Remark \ref{Qij(xik)=0} for $p=k=2$. The complex case is analogous.
\end{remark}

All the previous results can be summarized in the following.
\begin{thm}\label{Teorema principal de Sbrana Cartan}
    Let $f:M^n\rightarrow \mathbb{R}^{n+1}$ be a generic simply connected hypersurface of rank $2\leq p+1<n$. Suppose that $f$ posses a non-degenerate deformation in codimension $p$. Then the Gauss map $h:L^{p+1}\rightarrow\mathbb{S}^n$ possesses a conjugate chart with $\mathcal{S}^{*}\neq\emptyset$ and the support function satisfies $Q(\gamma)=0$. Moreover, the set $\mathcal{S}_\mu^{*}\subseteq\mathcal{S}^{*}$ naturally parametrizes the moduli space of non-degenerate genuine deformations of $f$ in $\mathbb{R}^{n+p}_{\mu}$. 
    
    Conversely, any pair $(h,\gamma)$ satisfying these properties is the Gauss data of a hypersurface $f:M^n\rightarrow\mathbb{R}^{n+1}$ which possesses non-degenerate genuine deformations in codimension $p$. 
\end{thm}
\begin{remark}
    In the converse, the parametrized hypersurface may not be generic and then the set $\mathcal{S}^{*}$ parametrize the non-degenerate deformations such that $\beta$ is diagonalizable by the vectors $X_i\in\Gamma_{\mathbb{C}}^\perp$ given by (\ref{Bajar coordenadas}). To verify if $f$ is generic, we express the splitting tensor in terms of the Gauss data. Using \cite{DG}, we see that for $(y,w)\in M^n=T^\perp_h L$, the splitting tensor is given by $C_\xi=B_\xi P_w^{-1}$ where $\xi\in T^\perp_hL(y)=\Delta(y,w)$ and $P_w$ was defined in (\ref{P_w}). Thus, the hypersurface is generic precisely in the open subset
    $$U=U_{h,\gamma}=\{w\in T^{\perp}_hL: P_w\text{ is invertible and }\exists\xi \text{ such that } B_\xi P_w^{-1} \text{ is semisimple over }\mathbb{C} \}.$$
\end{remark}

\subsection{The moduli space \texorpdfstring{$\mathcal{S}^{*}$}{TEXT}}\label{seccion especie}

In this subsection we introduce the notion of species of a conjugate chart. This concept will characterize $\mathcal{S}^{*}$ and also give a geometric description of it.

\text{ }

Suppose that $h:L^{p+1}\rightarrow\mathbb{S}^n$ has a conjugate chart $(u_0,\ldots,u_p)$. By Corollary \ref{corolario f} any section $\varphi$ over the trivial $\mathbb{C}$-bundle $\mathbb{C}^{p+1}\rightarrow L^{p+1}$ that is also in $\mathcal{S}^{*}$ must satisfy that
\begin{equation*}
    d\varphi+\omega\varphi=0,
\end{equation*}
where $\omega:TL\rightarrow\text{End}(\mathbb{C}^{p+1})$ is the bundle map $\omega_i(e_j):=\omega(\partial_i)(e_j)=\sum_k\omega_{ij}^ke_k$, where
\begin{equation}\label{omega de la coneccion}
    \omega_{ij}^k=\left\{ \begin{array}{lr}
         -2\Gamma_{ij}^j&  \text{if } k=j\neq i,\\
         2\Gamma_{ij}^j&  \text{if } k=i\neq j,\\
         0 &\,\text{ in other case}.
    \end{array}\right.
\end{equation}
In other words, this element $\varphi\in\mathcal{S}^{*}$ is a parallel section of the connection $\tilde{\nabla}:\mathfrak{X}(L^{p+1})\times\Gamma(\mathbb{C}^{p+1})\rightarrow\Gamma(\mathbb{C}^{p+1})$ over the trivial bundle given by
\begin{equation}\label{definicion de la conexion en el trivial}
    \tilde{\nabla}\xi=d\xi+\omega\xi.
\end{equation}
Notice that the conjugation $C(e_i)=e_{\overline{i}}$ is parallel with respect to $\tilde{\nabla}$ since $C$ commutes with $\omega_i$ for all $i$. This motivates the following definition.

\begin{defn}\label{sbrana bundle}
    \normalfont Consider a DMZ system 
    $$Q_{ij}=\partial_{ij}^2-\Gamma_{ji}^i\partial_i-\Gamma_{ij}^j\partial_j+g_{ij},\quad\forall 0\leq i<j\leq p,$$
    defined on $L^{p+1}\subseteq\mathbb{R}^{p+1}$. We call the real affine bundle $F=(\text{Re}_C(\mathbb{C}^{p+1})\rightarrow L^{p+1}, \tilde{\nabla}=d+\omega)$ the {\it Sbrana bundle associated to $Q$}, where $\omega$ is defined by (\ref{omega de la coneccion}).
\end{defn}
\begin{remark}
    Whenever $h:L^{p+1}\rightarrow\mathbb{S}^n$ has a conjugate coordinate system, that is $Q(h)=0$, then the Sbrana bundle is assumed to be associated to this DMZ system $Q$. 
\end{remark}

Any parallel section $\varphi$ of the Sbrana bundle satisfies that $\partial_i\big(\sum_j\varphi_j\big)=0$ for any $i$, so the sum of the coordinates is constant. Thus, if $\varphi(q)$ is admissible and has index $\mu$ for some $q\in L^{p+1}$, then $\varphi\in\mathcal{S}_\mu^{*}$ in the neighbourhood of $q$ where $\varphi_i\neq 0$ for all $i$.

For completeness, we describe next a procedure to find all parallel sections of an affine bundle $E$ namely, its trivial holonomy, as an integration of the Ambrose-Singer Theorem.
Since this result is local, we fix a trivialization and assume that $E=\mathbb{C}^{N}\rightarrow L^{p+1}$.
Denote by
\begin{equation*}
    \omega=\tilde{\nabla}-d \in\Gamma(T^{*}L\otimes \text{End}(\mathbb{C}^N)),
\end{equation*}
the connection 1-form and
\begin{equation*}
    \Omega_0:=\Omega=d\omega+[\omega,\omega]\in\Gamma(T^{*}L\otimes T^{*}L\otimes\text{End}(\mathbb{C}^N)),
\end{equation*}
the curvature 2-form. Fix any connection $\nabla$ for $L^{p+1}$, and define inductively
\begin{equation*}
    \Omega_k=\nabla\Omega_{k-1}-\Omega_{k-1}\circ\omega\in\Gamma\Big((\bigotimes_{k=0}^{k+1}T^{*}L)\otimes\text{End}(\mathbb{C}^N)\Big).
\end{equation*} 
Consider the sets
\begin{equation*}
    \Delta_{k}:=\{\varphi\in E:\Omega_k(X_0,\ldots,X_k)\varphi=0,\forall X_i\in TL\},
\end{equation*}
\begin{equation*}
    \mathcal{N}_k=\bigcap_{j=0}^k\Delta_{j}.
\end{equation*}
As usual we assume that $\mathcal{N}_k$ is a smooth vector bundle of $E$ for $k=0,\ldots,(N-1)$ since this is true along each connected component of an open dense subset of $L^{p+1}$.
\begin{prop}
    Assume that $\mathcal{N}_k$ is a smooth subbundle of $E=(\mathbb{C}^N\rightarrow L^{p+1},\tilde{\nabla})$ for $k=0,\ldots,N-1$. Then $\mathcal{N}_{N-1}$ is the maximal parallel flat subbundle of $E$. In particular, given any initial condition $\varphi_q\in \mathcal{N}_{N-1}(q)$ for some $q\in L^{p+1}$, there exists a unique parallel section $\varphi$ of $E$ such that $\varphi(q)=\varphi_q$ and $\varphi\in\Gamma(\mathcal{N}_{N-1})$. 
\end{prop}
\begin{proof}
    Suppose that $\varphi\in\Gamma(E)$ is a parallel section. Then as 
    $$0=d(d\varphi+\omega\varphi)=d(\omega\varphi)=(d\omega+[\omega,\omega])\varphi,$$
    we have that $\varphi\in\mathcal{N}_0$. If $\varphi\in\mathcal{N}_{k-1}$, then 
    \begin{align*}
        0&=\nabla_{X_{k+1}}(\Omega_{k-1}(X_0,\ldots,X_k)\varphi)=\nabla_{X_{k+1}}\Omega_{k-1}(X_0,\ldots,X_k)\varphi-\Omega_{k-1}(X_0,\ldots,X_k)\circ\omega(X_{k+1})\varphi\\
        &=\Omega_k(X_0,\ldots,X_{k+1})\varphi,\quad\forall X_i \in\Gamma(TL),
    \end{align*}
    which proves that $\varphi\in \mathcal{N}_k$ and inductively $\varphi\in\mathcal{N}_j$ for all $j$. Thus, any parallel flat subbundle is contained in $\mathcal{N}_{N-1}$. In particular, if $\mathcal{N}_{N-1}=0$ there are no non-trivial parallel sections.
    
    Assume that $\mathcal{N}_{N-1}\neq 0$, and consider the inclusions of $\mathbb{C}$-vector bundles 
    \begin{equation*}
        0\neq\mathcal{N}_{N-1}\subseteq\ldots\subseteq\mathcal{N}_0\subseteq\mathbb{C}^N=:\mathcal{N}_{-1}.
    \end{equation*}
    Let $k\in\{0,\ldots,N-1\}$ be the first index such that $\mathcal{N}_k=\mathcal{N}_{k-1}$. If $k=0$ this means that $E$ is flat and all flat bundles possess (local) parallel sections given any initial condition. Notice that in this case $\mathcal{N}_{N-1}=\mathcal{N}_0=\mathbb{C}^N$ since $\Omega_j=0$ for all $j$. Assume that $k\geq 1$. For any section $\xi$ of $\mathcal{N}_{k-1}$ and any $1\leq j\leq k$ we have that 
    $$0=\nabla(\Omega_{j-1}\xi)-\Omega_j\xi=\Omega_{j-1}\circ(\nabla\xi+\omega\xi),$$
    which shows that $\tilde{\nabla}\xi=\nabla\xi+\omega\xi\in\Gamma(T^{*}L\otimes \Delta_{j-1})$. Hence $\tilde{\nabla}\xi\in\Gamma(T^{*}L\otimes \mathcal{N}_{k-1})$, but by the choice of $k$, this proves that $\mathcal{N}_k\subseteq E$ is a parallel subbundle and then $\mathcal{N}_k\subseteq\mathcal{N}_{N-1}$ by the maximality property. Therefore $\mathcal{N}_k=\mathcal{N}_{N-1}$ is a flat parallel subbundle, which concludes the proof. 
\end{proof}

Using the above for the connection (\ref{definicion de la conexion en el trivial}), we can give a description of the moduli space $\mathcal{S}^{*}$. First, we notice that for $i\neq j$ the $i^{th}$-row of $\Omega_0(\partial_j,\partial_i)$ is the same as its $j^{th}$-row up to sign, and the remaining rows are zero. Thus, we can collect the non-trivial information of $\Omega_0$ in a single matrix. Let  $B:\mathbb{C}^{p+1}\rightarrow\mathbb{C}^{\binom{p+1}{2}}$ whose coefficients for $0\leq i<j\leq p$ are given by
$$B_{ijk}=\partial_i\Gamma^k_{jk}+2\Gamma^k_{ik}\Gamma^k_{jk}-2\Gamma^k_{ik}\Gamma^i_{ji}-2\Gamma^k_{jk}\Gamma^j_{ij} \text{ for }k\notin\{i,j\},$$
$$B_{iji}=\partial_i\Gamma^i_{ji}-2\Gamma^i_{ji}\Gamma^j_{ij},$$
$$B_{ijj}=\partial_j\Gamma^j_{ij}-2\Gamma^i_{ji}\Gamma^j_{ij}.$$
Then the $i^{th}$-row of $\Omega(\partial_j,\partial_i)$ is $2B_{ij}$ for $i<j$. Notice that the last two coefficients are precisely the ones that appear in the Sbrana-Cartan classification, yet the first one is new. In the same way as before, to $\Omega_k$ we can associate a matrix $B_k$ which contains its non-trivial data. Let $B_0=B$ and inductively
\begin{equation*}\label{definicion inductiva de B_n}
    B_{n+1}=\begin{pmatrix}
        \partial_0B_n-B_n\omega_0\\
        \vdots\\
        \partial_pB_n-B_n\omega_p
        \end{pmatrix}:\mathbb{C}^{p+1}\rightarrow\mathbb{C}^{\binom{p+1}{2}(p+1)^{n+1}}.
\end{equation*}
We conclude that
$$\mathcal{N}_p=\bigcap_{i=0}^p\ker(B_i).$$
Notice that the conjugation $C(e_i)=e_{\overline{i}}$ is parallel with respect to this connection and then, $\hat{\mathcal{N}}_{p}=\text{Re}_C(\mathcal{N}_{p})$ is the maximal parallel flat subbundle of the Sbrana bundle, i.e., its trivial holonomy.
\begin{defn}
    \normalfont Let $h:L^{p+1}\rightarrow\mathbb{S}^{n}$ be a submanifold with a conjugate chart.
    We say that $h$ is of the $k^{th}$-{\it species} for $1\leq k\leq p+1$ if the trivial holonomy of the Sbrana bundle $\hat{\mathcal{N}}_p\subseteq F$ has rank $(p+2-k)$ and is generic in the sense that intersects the open dense subset $\{v\in F: v_i\neq 0\text{ and }\sum_iv_i\neq 0\}\subseteq F$.
\end{defn}
\begin{remark}
    Our definition of species has a slight difference with the one in the Sbrana-Cartan to include semi-Riemannian ambient spaces. 
    The condition that $\tau$ in (\ref{solucion segunda especie sbrana cartan clasico}) has to be positive when the conjugate directions are real guarantees that the unique element in $\mathcal{U}$ has index $0$, in order to obtain a deformation on the Euclidean space, and not in the Lorentz space.
\end{remark}

\begin{remark}
    Given $h:L^{p+1}\rightarrow\mathbb{S}^n$ a submanifold with a conjugate chart, we call $L^j$ a {\it slice} of $L^{p+1}$ if $L^j$ is obtained after fixing some of the conjugate coordinate variables to some values. 
    In this case, the slice naturally has a conjugate chart for $H=h|_{L^j}$ by restricting the original coordinates to $L^j$.
    If $h:L^{p+1}\rightarrow\mathbb{S}^{n}$ is of $k^{th}$-species, with $\min\{2,k\}<p+1$, generically we can construct new submanifolds of some species by taking slices. Indeed, let $L^j\subseteq L^{p+1}$ a slice with $k\leq j$, then the trivial holonomy of the Sbrana bundle of $H=h|_{L^j}$ is at most $(p+2-k)$. Indeed, the rank of the Sbrana bundle of $H$ is $(p+2-k)$ if and only if the matrix
\begin{equation*}\label{mathcal B}
    \mathcal{B}=\mathcal{B}_{L^{p+1}}=(B_0^T\, B_1^T\,\ldots\, B_p^T)^T,
\end{equation*}
has rank $(k-1)$. Notice that the matrix $\mathcal{B}_{L^j}$ appears as a submatrix of the original $\mathcal{B}_{L^{p+1}}$, so it has less or equal rank. The condition of the trivial holonomy being generic is generically satisfied and in that case, $H$ is of $l^{th}$-species for some $l\leq k$.
\end{remark}

Assume now that $h$ is of the $k^{th}$-species for $1\leq k\leq p+1$, fix $q\in L^{p+1}$ and let
\begin{equation}\label{mathcal U}
    \mathcal{U}=\{u\in\hat{\mathcal{N}}_p(q):u\text{ is admissible}\}\subseteq\hat{\mathcal{N}}_p(q)\cap\{u=(u_i)_i:1+\sum_iu_i=0\}\cong\mathbb{R}^{p+1-k}.
\end{equation}
We have the natural bijection $u\rightarrow\varphi_u$ between $\mathcal{U}$ and $\mathcal{S}^{*}$, where $\varphi_u\in\Gamma(\hat{\mathcal{N}}_p)$ is the parallel section which satisfies $\varphi_u(q)=u$.
Naturally, the open subset 
\begin{equation}\label{mathcal Ur}
    \mathcal{U}_\mu=\{u\in\mathcal{U}:u\text{ is admissible and has index }\mu\}\subseteq\mathcal{U},
\end{equation}
is in bijection with $\mathcal{S}^{*}_\mu$.
We conclude:
\begin{thm}\label{S* es un abierto de R}
    Suppose that $h:L^{p+1}\rightarrow\mathbb{S}^{n}$ is of $k^{th}$-species for some $k\in\{1,\ldots,p+1\}$. Then $\mathcal{S}^{*}$ and $\mathcal{S}_\mu^{*}$ are naturally diffeomorphic to a finite union of open and convex subsets of $\mathbb{R}^{p+1-k}$ for all $0\leq\mu\leq p$. Moreover, $\mathcal{S}_0^{*}\cong\mathcal{U}_0\subseteq\mathbb{R}^{p+1-k}$ has at most $(p+1)$ connected components.
\end{thm}
\begin{proof}
    By the above discussion, we only need to bound the number of connected components of $\mathcal{S}_0^{*}$. Proposition~\ref{singular signature es el indice de la metrica} bounds the number of connected components of the set $U=\{\varphi:\varphi\text{ is admissible of index } 0\}$. 
    If the conjugate chart has a complex conjugate chart this set is convex. If the conjugate coordinates are real then $U$ has $(p+1)$ convex components determined by the choice of which coordinate is negative. Thus, $\mathcal{U}_0=U\cap\mathcal{N}_p$ has at most $(p+1)$ components that are convex since each one is an intersection of convex subsets.
\end{proof}

In order to recover the discrete and continuous types of hypersurfaces in the Sbrana-Cartan classification we introduce the following concept. The last remark also let us bound the number of connected components.
\begin{defn}\label{definicion tipo}
    \normalfont We say that a generic hypersurface $f:M^n\rightarrow\mathbb{R}^{n+1}$ of rank $2\leq p+1<n$ is of the {\it $r^{th}$-type}, for $r\in\{0,\ldots, p\}$ if the set of genuine deformations $g:M^n\rightarrow\mathbb{R}^{n+p}$ is naturally an union of at most $(p+1)$ convex open subsets of $\mathbb{R}^{r}$.
\end{defn}

Finally, we can prove Theorem \ref{teorema introduccion}.
\begin{proof}[Proof of Theorem \ref{teorema introduccion}]\label{prueba toerema intro}
    As discussed in the preliminaries, any hypersurface $f:M^n\rightarrow\mathbb{R}^{n+1}$ of rank $2<p+1<n$ is genuinely rigid in $\mathbb{R}^{n+q}$ for any $q<p$ if we add the hypotheses of not being $(n-p+3)$-ruled for $p\geq 7$. To conclude the proof, by Theorem \ref{Teorema principal de Sbrana Cartan} and Theorem \ref{S* es un abierto de R} we only need to show that any genuine deformation $g:M^n\rightarrow\mathbb{R}^{n+p}$ of $f$ is non-degenerate.

    First, observe that the cases $p\leq 4$ and $p\geq 7$ are immediate by Corollary \ref{genuinas son no degeneradas}.
    
    For $p\in\{5,6\}$ we use Remark \ref{p=5,6}. Assume that $\mathcal{S}(\beta)$ degenerates. Then by Remark \ref{p=5,6}, $\Delta_g=\Gamma\subsetneq R^d$, where $R^d$ is some mutual ruling for $f$ and $g$. Denote $\tilde{R}=R\cap\Gamma^\perp$. As $R^d$ is totally geodesic, $C_T(\tilde{R})\subset\tilde{R}$  for all $T\in\Gamma$, and then by the generic condition we get $X_i\in\tilde{R}$ where $X_i$ is some eigenvector of the semisimple endomorphism $C_{T_0}$. However, this implies that $X_i\in\Gamma$, since the eigenvectors of $C_{T_0}$ diagonalize $\beta$ by (\ref{beta y CT}) and $\beta(X_i,X_i)=0$ as $X_i\in R^d$, which is a contradiction. Thus, $\mathcal{S}(\beta)$ is non-degenerate. The Main Lemma and Corollary 2 of \cite{Moo} imply that $g$ is non-degenerate.
\end{proof}

\section{Deformations of generic hypersurfaces in codimension 2}\label{seccion deformaciones en codim 2}
In this section we apply Theorem \ref{teorema introduccion} to prove Theorem \ref{Descripcion SC p=2 generica}. It is an analogous description to the one given by Sbrana and Cartan, and characterizes all the deformable generic hypersurfaces $f:M^n\rightarrow\mathbb{R}^{n+1}$ in codimension 2 and its moduli space of deformations. As already observed, the hypothesis of being generic is to discard the surface-like and ruled type of situation. 

\text{ }

We start by recalling Sbrana-Cartan hypersurfaces of {\it intersection type}, as named in \cite{FF}. 
They are Rimeannian submanifolds $M^n$ obtained by intersecting two flat hypersurfaces $F:U_1\subseteq\mathbb{R}^{n+1}\rightarrow \mathbb{R}^{n+2}_\nu$ and $G:U_1\subseteq\mathbb{R}_{\mu}^{n+1}\rightarrow \mathbb{R}^{n+2}_\nu$ in general position. Then 
$$M^n=F_1(U_1)\cap F_2(U_2)\subseteq\mathbb{R}_{\nu}^{n+2},$$ 
$f$, $g$ stands for the inclusions of $M^n$ into $U_1$ and $U_2$ respectively, and $H:=F\circ f= G\circ g$. They were introduced in \cite{DFT2} for $(\mu,\nu)=(0,0)$ and studied in \cite{DF6} for $(\mu,\nu)=(1,1)$. The case $(\mu,\nu)=(0,1)$ is new and necessary to present the deformations of hypersurfaces in codimension $2$. 

A hypersurface $f:M^n\rightarrow\mathbb{R}^{n+1}$ of intersection type is determined by the conjugate chart $(u,v)\in\mathbb{R}^2$ of its Gauss map $h:L^2=M^n/\Gamma\rightarrow\mathbb{S}^n$. In fact, the Christoffel symbols satisfy 
\begin{equation}\label{equacion que caracteriza las intersecciones}
    \partial_v\Gamma_{uv}^v-\Gamma_{vu}^u\Gamma_{uv}^v+g_{uv}=0.
\end{equation}
Namely, if $Q$ is the hyperbolic linear operator 
$$Q:=\partial_{uv}^2-\Gamma_{vu}^u\partial_u-\Gamma_{uv}^v\partial_v+g_{uv},$$
for which $Q(h)=Q(\gamma)=0$ where $\gamma$ is the support function of $f$, then one of its Laplace invariants vanishes.
Moreover, if (\ref{equacion que caracteriza las intersecciones}) holds, then any non-degenerate deformation of $f$ is obtained as an intersection. 
In fact, in \cite{DFT2} they show that if $g$ is any such deformation of $f$ given by $\varphi=(\varphi_0,\varphi_1)\in\mathcal{S}^{*}$ with $\varphi_1<-1$ then the index of $\varphi$ is $\mu=0$ and the intersection is in $\mathbb{R}^{n+2}$. 
If $\varphi_1\in(-1,0)$ then the index of $\varphi$ is $\mu=1$ and they intersect in $\mathbb{R}_1^{n+2}$ as in \cite{DF6}. 
Similarly, if $\varphi_1>0$ then the index of $\varphi$ is $\mu=0$ and the intersection is in $\mathbb{R}_1^{n+2}$.

By Theorem 1 of \cite{DF3}, in order for a generic hypersurface $f:M^n\rightarrow \mathbb{R}^{n+1}$ to have a genuine deformation in codimension $2$, its rank must be at most $3$. If it is less than $2$, then $M^n$ is flat, and all the local immersion are described in Corollary 18 of \cite{FF}. Theorem \ref{teorema introduccion} characterizes the rank $3$ case. Theorem 1 of \cite{DFT} describes when the rank is $2$, but this result has a gap that we discuss next.

If $f:M^n\rightarrow \mathbb{R}^{n+1}$ is a Sbrana-Cartan hypersurface, $g:M^n\rightarrow U\subseteq\mathbb{R}^{n+1}$ a genuine deformation of $f$ and $j:U\subseteq\mathbb{R}^{n+1}\rightarrow\mathbb{R}^{n+2}$ an isometric immersion with $\alpha^j\neq 0$, then $\hat{g}=j\circ g$ is generically also a genuine deformation of $f$ which is not considered in Theorem 1 of \cite{DFT}.
In particular, for rank two generic Sbrana-Cartan hypersurfaces, there are more genuine deformations than the moduli space $\mathcal{C}_h$ described in that paper. 
As defined in \cite{FF}, we say that a genuine deformation $g$ of $f$ is {\it honest} if $g$ is not a composition as before.
The set $\mathcal{C}_h$ measures the honest deformations of $f$ except for Sbrana-Cartan hypersurfaces of intersection type. For such hypersurfaces, some deformations described by $\mathcal{C}_h$ are not honest. 
Indeed, let $\hat{g}:M^{n}\rightarrow\mathbb{R}^{n+2}$ be a genuine deformation of a rank $2$ generic hypersurface $f$ associated with some element in $\mathcal{C}_h$ and assume that $\hat{g}=j\circ g$ for some isometric immersions $g:M^{n}\rightarrow U\subseteq\mathbb{R}^{n+1}$ and $j:U\subseteq\mathbb{R}^{n+1}\rightarrow\mathbb{R}^{n+2}$. If $\Gamma\subseteq TM$ is elliptic (that is, some splitting tensor $C_T$ has non-real eigenvalues), then
$$\alpha^j(u,u)+\alpha^j(v,v)=0,$$ 
for some basis $u,v\in\Gamma^\perp$. This and the flatness of $\alpha^j$ imply that $\alpha^{\hat{g}}=\alpha^g$. 
This is a contradiction since the deformations described by $\mathcal{C}_h$ satisfy that $\dim(\mathcal{S}(\alpha^{\hat{g}}))=2$.
Then $\Gamma\subseteq TM$ is hyperbolic (that is, some splitting tensor $C_T$ is semisimple over $\mathbb{R}$), and let $(u,v)\in\mathbb{R}^2$ be the conjugate chart of the Gauss map $h:L^2=M^n/\Gamma\rightarrow\mathbb{S}^n$ of $f$ satisfying (\ref{Bajar coordenadas}) for $X_0, X_1\in\Gamma^{\perp}$ the eigenvectors of the splitting tensors of $\Gamma\subseteq TM$. Then $g$ and ${\hat{g}}$ are genuine deformations of $f$ associated to some $\varphi=(\varphi_0,\varphi_1)\in\mathcal{S}^{\#}$ and $(U,V)\in\mathcal{C}_h$ respectively. Define
$$\hat{\varphi}_0(u,v):=2U(u)e^{\int_0^{v}2\Gamma_{vu}^{u}(u,s)ds},\quad\hat{\varphi}_1(u,v):=2V(v)e^{\int_0^{u}2\Gamma_{uv}^{v}(s,v)ds}\quad\text{ and }\quad\hat{\varphi}_{*}:=-(1+\hat{\varphi}_0+\hat{\varphi}_1)\neq0.$$ 
By the definition of $\mathcal{C}_h$ we have that  $\hat{\varphi}=(\hat{\varphi}_0,\hat{\varphi}_1,\hat{\varphi}_{*})$ is admissible of index $0$ and $Q(\sqrt{|\hat{\varphi}_{*}|})=0$. 
Codazzi equation implies that
$$\alpha^g(X_0,X_1)=\alpha^{\hat{g}}(X_0,X_1)=\alpha^j(X_0,X_1)=0,$$ 
and for $i=0,1$, let
$$\hat{\eta}_i:=\frac{\alpha^{\hat{g}}(X_i,X_i)}{\langle AX_i,X_i\rangle}=\frac{\alpha^{g}(X_i,X_i)}{\langle AX_i,X_i\rangle}+\frac{\alpha^j(X_i,X_i)}{\langle AX_i,X_i\rangle}=:\eta_i+\varepsilon_i.$$
By flatness of $j$ and dimension reasons, we can assume that $\varepsilon_1\neq0$ and $\varepsilon_0=0$. Thus,
$$1+\frac{1}{\hat{\varphi}_0}=\langle\hat{\eta}_0,\hat{\eta}_0\rangle=\langle\eta_0,\eta_0\rangle=1+\frac{1}{\varphi_0}.$$
Here we used the geometric interpretation of $(U,V)\in\mathcal{C}_h$. 
Then $\hat{\varphi}_{*}=\varphi_1-\hat{\varphi}_1$ and by (\ref{derivadas de phi_i caso p+1}) we have that $\partial_u\hat{\varphi}_{*}=2\Gamma_{uv}^v\hat{\varphi}_{*}$, but in this case, 
\begin{equation*}
    0=Q(\sqrt{|\hat{\varphi}_{*}|})=(\partial_v\Gamma_{uv}^v-\Gamma_{vu}^u\Gamma_{uv}^v+g_{uv})\sqrt{|\hat{\varphi}_{*}|}.
\end{equation*}
Thus, all the genuine deformations described by $\mathcal{C}_h$ are honest except when (\ref{equacion que caracteriza las intersecciones}) is satisfied, that is, when the hypersurface is of intersection type.

Theorem 1 of \cite{DFT} for Sbrana-Cartan hypersurfaces of intersection type only says that the moduli space of honest deformations is a subset of $\mathcal{C}_h$. 
However, Theorem 33 of \cite{FF} classifies all the honest deformations in codimension 2 for hypersurfaces obtained as intersections in $\mathbb{R}^{n+2}$. Thus, we need to extend some concepts and results of \cite{FF} to describe the honest deformations of Sbrana-Cartan hypersurfaces which are intersections in $\mathbb{R}^{n+2}_1$. Almost all the ideas are analogous, so we will leave the details to the reader.

Let $H:M^{n}\rightarrow\mathbb{R}^{n+2}_\mu$ be a generic Riemannian submanifold of rank $2$, $\mathcal{S}(\alpha^H)=T^\perp_HM$. Then we can construct a polar surface in a similar way as in \cite{FF} or \cite{DF7}.
If $\Delta_H\subseteq TM$ is hyperbolic (the eigenvectors of the splitting tensors are real), then the polar surface is an immersion $g:L^2=M^n/\Delta_H\rightarrow\mathbb{R}^{n+2}_1$ such that $g_{*}(TL)=T^{\perp}_HM$ and has conjugate coordinates $(u,v)\in\mathbb{R}^2$. Namely, it satisfies a hyperbolic linear differential equation
$$\tilde{Q}(g)=\partial^2_{uv}g-\tilde{\Gamma}_{vu}^u\partial_ug-\tilde{\Gamma}_{uv}^v\partial_vg=0.$$

Let $f:M^n\rightarrow\mathbb{R}^{n+1}$ be a Sbrana-Cartan hypersurface obtained as an intersection of two flat Riemannian hypersurfaces of $\mathbb{R}^{n+2}_1$. The inclusion $H:M^n\rightarrow\mathbb{R}^{n+2}_1$ satisfies $\Delta_H=\Gamma$. Thus it has a polar surface. 
Moreover, as discussed in Section 9 of \cite{FF}, this surface is the sum of two curves
$$g(u,v)=\alpha_1(u)+\alpha_2(v),$$
with $\alpha_1', \alpha_1'', \alpha_2', \alpha_2''$ being pointwise linearly independent, $\langle\alpha_1',\alpha_1'\rangle=\langle\alpha_2',\alpha_2'\rangle=-1$ and $\cosh(\theta):=-\langle\alpha_1',\alpha_2'\rangle$. 
This characterizes the hypersurfaces of intersection type obtained as the intersection of two Riemmanian flat hypersurfaces in $\mathbb{R}^{n+2}_1$. 
Similarly to Theorem 32 of \cite{FF}, the Sbrana-Cartan hypersurface of intersection type is of discrete type if $I(H)\geq 2$ and continuous if $I(H)=1$, where $I(H):=I(\alpha_1,\alpha_2)$ is the shared dimension of $\alpha_1$ and $\alpha_2$ as defined in Section \ref{seccion shared dimension}. 

The following result is an adaptation of Theorem 33 of \cite{FF} for Lorentz ambient space.

\begin{thm}\label{deforamciones de las intersecciones en lorentz}
    Let $f:M^n\rightarrow\mathbb{R}^{n+1}$ be a Sbrana-Cartan hypersurface obtained as an intersection of two flat Riemmanian hypersurfaces of $\mathbb{R}^{n+2}_1$, and let $H:M^n\rightarrow\mathbb{R}^{n+2}_1$ be the inclusion. Then $f$ is honestly rigid in $\mathbb{R}^{n+2}$, unless $I(H)=2$. In the latter case, the moduli space of honest deformations is an open interval of $\mathbb{R}$.
\end{thm}
\begin{proof}
    Since the proof is analogous to Theorem 33 of \cite{FF}, we will only point out the slight differences.
    Using the notations in \cite{FF}, we have in particular $s:=-\sinh(\theta)^2$, and our analogous functions $U=U(u)$ and $V=V(v)$ must satisfy
    \begin{equation}\label{desigualdades de U y V}
        U,V>-\frac{1}{s},\quad\text{ and }\quad (U+1)(V+1)<(1-s)UV.
    \end{equation}
    
    The hypersurface $f$ is honestly rigidity in $\mathbb{R}^{n+2}$ for $I(H)\neq 2$, for analogous reasons. 
    When $I(H)=2$, \cite{FF} uses the geometric characterization of this index to project into the shared space, which may not be possible in Lorentz ambient space. However, since $\text{span}(\alpha_1)$, $\text{span}(\alpha_2)\subseteq\mathbb{R}^{n+2}_1$ are Lorentzian subspaces, let $\mathbb{V}^l\subseteq\mathbb{R}^{n+2}_1$ be the Lorentz subspace given by Lemma \ref{lema shared dimension}, with $l\leq 2$. If $l=1$ then $I(H)=1$, so $l=2$. 
    Define $\overline{\alpha}_i$ as the orthogonal projection of $\alpha_i$ in $\mathbb{V}^2$ for $i=1,2$. 
    Then $\overline{\alpha}_1$, $\overline{\alpha}_2$ are light-like curves of $\mathbb{V}^2$ and 
    \begin{equation}\label{dsds}
        \langle\overline{\alpha}'_1,\overline{\alpha}'_1\rangle \langle\overline{\alpha}'_2,\overline{\alpha}'_2\rangle< \langle\overline{\alpha}'_1,\overline{\alpha}'_2\rangle^2=\langle\alpha'_1,\alpha'_2\rangle^2=\cosh(\theta)^2=1-s.
    \end{equation}
    Here we used the Cauchy Schwarz inequality for time-like vectors. Those curves work as the curves defined in \cite{FF} with the same notations.
    
    Following the steps in the proof given in \cite{FF}, we see that the moduli space of honest deformations is in bijection with 
    $$(U_t,V_t)=\big((t^{-1}\langle\overline{\alpha}'_1,\overline{\alpha}'_1\rangle-1)^{-1},(t\langle\overline{\alpha}'_2,\overline{\alpha}'_2\rangle-1)^{-1}\Big),$$
    for $0\neq t\in\mathbb{R}$ such that (\ref{desigualdades de U y V}) is satisfied. That and (\ref{dsds}) give us that $t$ must satisfy
    \begin{equation}\label{intervalo del t}
        \langle\overline{\alpha}'_1(u),\overline{\alpha}'_1(u)\rangle<t<\langle\overline{\alpha}'_2(v),\overline{\alpha}'_2(v)\rangle^{-1}.
    \end{equation}
    This is possible since $\langle\overline{\alpha}'_1,\overline{\alpha}'_1\rangle\langle\overline{\alpha}'_2,\overline{\alpha}'_2\rangle>\langle\alpha'_1,\alpha'_1\rangle\langle\alpha'_2,\alpha'_2\rangle=1$. If $t$ satisfies the above inequality for $(u,v)=(u_0,v_0)$, then (\ref{intervalo del t}) holds for $(u,v)$ in a neighborhood of $(u_0,v_0)$. Hence the honest deformations in $\mathbb{R}^{n+2}$ are in natural bijection with the open subset $\Big(\langle\overline{\alpha}'_1(u_0),\overline{\alpha}'_1(u_0)\rangle,\langle\overline{\alpha}'_2(v_0),\overline{\alpha}'_2(v_0)\rangle^{-1}\Big)\subseteq\mathbb{R}$.
\end{proof}

\section{Riemannian hypersurfaces in Lorentz ambient space}\label{Seccion lorentzianas}
All our analysis above can be translated for Riemannian hypersurfaces of the Lorentz space, that is, for generic hypersurfaces $f:M^n\rightarrow\mathbb{R}_1^{n+1}$ of rank $(p+1)$. In this subsection we provide some remarks about this together with an application for studying conformally flat Euclidean submanifolds. As the analysis is similar to the Euclidean case, we leave the details to the reader.

\text{ }

Analogously to the Euclidean case, there is a Gauss parametrization $(h,\gamma)$ for Riemannian submanifolds $F:M^n\rightarrow\mathbb{R}_1^{n+1}$ of rank $(p+1)$, where $h:L^{p+1}\rightarrow\mathbb{H}^n$ and $\gamma:L\rightarrow\mathbb{R}$ (see \cite{DG}). This parametrization can be used in the same way as before to study deformations of Lorentzian hypersurfaces.

Suppose that there is a non-degenerate deformation $g:M^n\rightarrow\mathbb{R}^{n+p}_{\mu}$ of $f$. If $M^n$ is generic, then we can define $\varphi_i$ and $\eta_i$ as in (\ref{definion de tau_i}) and (\ref{eta_i}), but in this case $\langle \eta_i,\eta_j\rangle=-(1+\frac{\delta_{ij}}{\varphi_i}),$
instead of (\ref{d_ij definicion}). This shows that the index of $\varphi$ is $p-\mu$, instead of being $\mu$ as in the Riemannian case. The diagonalizing directions also define a conjugate chart for $h:L^{p+1}\rightarrow\mathbb{H}^n$, in the same way as for submanifolds of the sphere, but in this case 
\begin{equation*}\label{Q_ij Lorentz}
    Q_{ij}(h)=\partial_{ij}^2h-\Gamma_{ji}^i\partial_ih-\Gamma_{ij}^j\partial_jh-g_{ij}h=0,\quad\forall i\neq j.
\end{equation*} 
We define $\mathcal{S}_{\mu}^{*}$ and $\mathcal{S}^{*}$ as in Definition \ref{definicion S y S*}. Theorem \ref{Teorema principal de Sbrana Cartan} holds as for Euclidean hypersurfaces, but $\mathcal{S}_{\mu}^{*}$ parametrize the non-degenerate deformations of $f$ in $\mathbb{R}^{n+p}_{p-\mu}$. Moreover, the concept of species can also be used to give an interpretation of $\mathcal{S}^{*}$.

This can be used to study conformally flat Euclidean submanifolds, namely, submanifolds $f:M^n\rightarrow\mathbb{R}^{n+p+1}$ that are conformally flat. It is known that a simply connected manifold $M^n$ with $n\geq 3$, is conformally flat if and only if it can be realized as a hypersurface of the light cone $V^{n+1}=\{X\in\mathbb{R}_1^{n+2}:\langle X,X\rangle=0, X\neq 0\}\subseteq\mathbb{R}_1^{n+2}$ (see for example \cite{AD}, \cite{DF2}). Thus, to obtain examples of conformally flat manifolds of $\mathbb{R}^{n+p+1}$ ($p\geq 1$), we can take a Riemannian manifold $N^{n+1}$ which has isometric immersions $F:N^{n+1}\rightarrow\mathbb{R}_1^{n+2}$ and $G:N^{n+1}\rightarrow\mathbb{R}^{n+p+1}$, and take $M^n$ as the intersection $F(N^{n+1})\cap V^{n+1}$ and $g=G|_{M^n}$. The first main result of \cite{DF2} states that this procedure generates all the simply connected examples for $p\leq n-4$. 

Consider $F:N^{n+1}\rightarrow\mathbb{R}_1^{n+2}$ a nowhere flat hypersurface of rank $(p+1)\geq 2$. Let $G:N^{n+1}\rightarrow\mathbb{R}^{n+q+1}$ be an isometric immersion. The Main Lemma for $\beta=\alpha^G\oplus\alpha^F$ proves that $q\geq p$, and if $q=p$ then $\mathcal{S}(\beta)=W=T^{\perp}_GM\oplus T^{\perp}_FM$. Assume that $q=p$. Notice that $G$ is always a non-degenerate deformation of $F$ since $W$ has positive signature. The techniques of this work can be used in this context analogously. In this case, the existence of the diagonalizing directions $X_i\in\Gamma^\perp$ for $\beta=\alpha^G\oplus\alpha^F:TN\times TN\rightarrow W^{p+1,0}$ comes from Theorem~2 of \cite{Moo}. Thus, the condition of being generic is not necessary in this context.

The proof of Theorem \ref{teorema introduccion} can be easily adapted to prove Theorem \ref{teorema de hipersuperficies lorentzianas}. When $N^{n+1}$ in Theorem \ref{teorema de hipersuperficies lorentzianas} is also generic, the conjugate chart is uniquely determined up to order and re-scaling factors of the basis. In this case, all the isometric immersions $G:N^{n+1}\rightarrow\mathbb{R}^{n+1+p}$ are in bijection some $\varphi=\varphi_G\in\mathcal{S}_p^{*}$. We can define the type of the hypersurface $F$ in the same way as in Definition \ref{definicion tipo}.
Thus, if the hypersurface $F$ is of the $r^{th}$-type, the set of such $G$'s is in bijection with $\mathcal{U}_p\subseteq\mathbb{R}^{r}$ as in (\ref{mathcal Ur}). In this case, $\mathcal{U}_p$ is actually diffeomorphic to $\mathbb{R}^{r}$. 
Indeed, since the index of $\varphi$ must be $0$, Proposition \ref{singular signature es el indice de la metrica} guarantees that $\varphi_i\in(-1,0)$ for all $i$, which is a convex set, thus $\mathcal{U}_p$ it is also convex.

\section{Appendix}\label{apendice}
In this section we prove some minor technical results used in this work.

\subsection{Description of an admissible \texorpdfstring{$\varphi$}{TEXT} and its index}
Here we characterize the property of a tuple $\varphi=(\varphi_i)_{i=0}^p$ being admissible (with respect to a basis $\{e_i\}_{i=0}^p$ of $\mathbb{W}_{\mathbb{C}}$ and a conjugation of indices $\overline{e_i}=e_{\overline{i}}$); see Definition \ref{definicion de ser singular}. This description relates $\varphi$ with a non-degenerate inner product and the index of $\varphi$ coincides with the index of such product. We assume that the first $2s$ coordinates are complex conjugate and the remaining are real.
\\\text{ }

Consider a tuple $\varphi=(\varphi_i)_{i=0}^p$ such that $\overline{\varphi_i}=\varphi_{\overline{i}}\neq 0$ for all $i$. Let $D_{\varphi}:\mathbb{W}_{\mathbb{C}}\rightarrow\mathbb{W}_{\mathbb{C}}$ the linear map defined by $D_\varphi(e_i)=\sum_{j}d_{ij}e_j$, where $d_{ij}=1+\frac{\delta_{ij}}{\varphi_i}$. Since $\overline{D_\varphi(e_i)}=D_\varphi(e_{\overline{i}})$, this linear map can be considered as a real one $D_\varphi:\mathbb{W}\rightarrow\mathbb{W}$.  An easy induction process proves the following.
\begin{lema}\label{Lema del determinante de D}
    Let $\varphi=(\varphi_i)_{i=0}^p$ is such that $\varphi_i\neq 0$ for all $i$, then
    \begin{equation*}
        \det(D_{\varphi})=\dfrac{1+\sum_i\varphi_i}{\Pi_i\varphi_i}.
    \end{equation*}
    If $\det(D_{\varphi})\neq 0$, then $D_{\varphi}^{-1}$ is given by
    \begin{equation*}
        (D_{\varphi}^{-1})_{ij}=\delta_{ij}\varphi_i-\frac{\varphi^2_i}{1+\sum_k\varphi_k}.
    \end{equation*}
\end{lema}

When $\varphi$ is admissible the above lemma implies that $D_\varphi$ has a kernel of dimension exactly $1$. Indeed, in this case the determinant of the minor of $D_\varphi$ obtained by deleting the $i^{th}$ row and column is $\frac{-\varphi_i}{\Pi_{j\neq i}\varphi_j}\neq 0$. Moreover, we can verify that
\begin{equation}\label{kernel de S}
    \ker(D_{\varphi})=\text{span}\bigg\{\sum_j\varphi_je_j\bigg\}.
\end{equation}
Therefore, when $\varphi$ is admissible, $D_\varphi$ induces an non-degenerate inner product on the $p$-dimmension real vector space $\mathbb{W}/\ker(D_\varphi)$ by the formula 
\begin{equation}\label{producto caso singular}
    \langle [e_i],[e_j]\rangle=d_{ij}=1+\frac{\delta_{ij}}{\varphi_i}\quad\forall i,j\in I,
\end{equation}
where $[e]:=e+\ker(D_\varphi)$.

\begin{prop}\label{singular signature es el indice de la metrica}
    If $\varphi$ is admissible then the index of $\varphi$ is precisely the index of the non-degenerate product given in (\ref{producto caso singular}).
\end{prop}
\begin{proof}
    Denote by $\mu$ the index of the product given by (\ref{producto caso singular}). Consider on $\mathbb{W}_{\mathbb{C}}$ the bilinear product $\langle e_i,e_j\rangle=\frac{\delta_{ij}}{\varphi_i}$. This defines a product on $\mathbb{W}=\text{Re}_{C}(\mathbb{W}_{\mathbb{C}})$, where $C$ denotes the conjugation given by the conjugation of indices. We identify the signature of this product in two ways. Let
    $$\xi_{2j}=\frac{1}{\sqrt{2}}\Big(\omega_je_{2j}+\overline{(\omega_j e_{2j})}\Big)\quad\text{ and }\quad\xi_{2j+1}=\frac{1}{\sqrt{2}}\Big(i\omega_j e_{2j}+\overline{(i\omega_je_{2j})}\Big)\quad\text{ for }0\leq j<s,$$
    where $\omega_j$ is any of the two complex roots of $\varphi_{2j}$. For $j\geq 2s$ define
    $$\xi_j=\omega_je_j,$$
    where $\omega_j$ is the positive root of $|\varphi_j|$. Then $\{\xi_j\}_{j=0}^p$ is an orthonormal basis of $\mathbb{W}$ of index $p+1-(s+P)$.
    
    Setting $\xi=\sum \varphi_je_j$ and $v_j=e_j+\xi$, then $\langle \xi,\xi\rangle=-1$, $\langle v_j,v_j\rangle=1+\frac{\delta_{ij}}{\varphi_i}=d_{ij}$ and $\langle \xi,v_j\rangle=0$. This gives us the orthogonal decomposition $\mathbb{W}=\text{Re}(\text{span}\{v_j\})\oplus\text{span}\{\xi\}$, and then the product has index $\mu+1$. Thus $\mu=p-(s+P)$.
\end{proof}

\subsection{The shared dimension of two curves}\label{seccion shared dimension}
In this subsection we extend the concept of shared dimension of two curves, which was introduced in \cite{FF} for the Euclidean ambient space, to the semi-Euclidean case. 

\text{ }

Given two curves $\alpha_i:I_i\subseteq\mathbb{R}\rightarrow\mathbb{R}^N_\mu$ ($i=0,1$) in a semi-Euclidean ambient space, we define the index $\overline{I}(\alpha_1,\alpha_2)$ as the minimum integer $k$ such that $\langle\alpha_1'(u),\alpha'(v)\rangle$ can be written as a sum $\sum_{j=1}^k a_j(u)b_j(v)$ for some smooth functions $a_j,b_j$, $1\leq j\leq k$. Let
$$I(\alpha_1,\alpha_2)(u,v)=\lim_{\varepsilon\to0}\overline{I}(\alpha_1|_{(u-\varepsilon,u+\varepsilon)},\alpha_2|_{(v-\varepsilon,v+\varepsilon)}),$$
which is semicontinuous and constant along connected components of an open dense subset of the parameters $(u,v)$. Following \cite{FF}, we call this integer the {\it shared dimension} between $\alpha_1$ and $\alpha_2$. For Euclidean ambient space, this agrees (locally) with the dimension of $\text{span}(\alpha_1)\cap\text{span}(\alpha_2)$, where $\text{span}(\alpha_i)$ is the smallest subspace which contains the image of the curve $\alpha_i$. This is not true for semi-Euclidean ambient spaces. 
If $\text{span}(\alpha_1)$, $\text{span}(\alpha_2)$ and $\text{span}(\alpha_1)\cap\text{span}(\alpha_2)$ are non-degenerate subspaces of $\mathbb{R}^{N}_\nu$, then clearly $$\dim\Big(\text{span}(\alpha_1)\cap\text{span}(\alpha_2)\Big)\geq I(\alpha_1,\alpha_2).$$

The following lemma allow us to decompose the ambient space in relation to the shared dimension. The proof is similar to the one of Lemma 10 in \cite{FF}.
\begin{lema}\label{lema shared dimension}
    Let $\alpha_1,\alpha_2$ curves in $\mathbb{R}^{N}_\nu$ such that $\big(\text{span}(\alpha_i)\big)^{\perp}\subseteq\mathbb{R}^N_\nu$ is a definite subspace for $i=1,2$ and
    $$\mathbb{U}:=\text{span}(\alpha_1)+\text{span}(\alpha_2)\subseteq\mathbb{R}^{N}_\nu,$$ 
    is non-degenerate. Then there exists an orthogonal decomposition $\mathbb{R}^{N}_{\nu}=\mathbb{V}_1\oplus\mathbb{V}^l\oplus\mathbb{V}_2$ such that $l\leq I(\alpha_1,\alpha_2)$, and $\text{span}(\alpha_i)\subseteq\mathbb{V}_i\oplus\mathbb{V}^l$, $i=1,2$. In particular, $\dim(\text{span}(\alpha_1)\cap\text{span}(\alpha_2))\leq I(\alpha_1,\alpha_2)$.
\end{lema}
\begin{proof}
    Clearly, we can assume that $\mathbb{U}=\mathbb{R}^N_{\nu}$. Write $\langle\alpha_1'(u),\alpha_2'(v)\rangle=\sum_{i=1}^ka_i(u)b_i(v)$, and set $$\hat{\alpha}_1(u)=\Big(\alpha_1(u),-\int_0^ua_1(s)ds,\ldots,-\int_0^ua_k(s)ds\Big),\quad\text{ and }\quad\hat{\alpha}_2(v)=\Big(\alpha_2(v),\int_0^vb_1(s)ds,\ldots,\int_0^vb_k(s)ds\Big),$$
    as orthogonal curves in $\mathbb{R}^{N+k}_{\nu}=\mathbb{R}^{N}_{\nu}\oplus\mathbb{R}^{k}_{0}$. Consider $\mathcal{E}=\text{span}(\hat{\alpha}_1)\cap\text{span}(\hat{\alpha}_2)\subseteq\mathbb{R}^{N+k}_{\nu}$ which is a null subspace. Then using a pseudo-orthogonal basis we can express $\mathbb{R}^{N+k}_{\nu}=\hat{\mathbb{V}}^{n_1}_1+\hat{\mathbb{V}}^{n_2}_2$, with $\hat{\mathbb{V}}^{n_1}_1$, $\hat{\mathbb{V}}^{n_2}_2$ orthogonal and $\hat{\mathbb{V}}^{n_1}_1\cap\hat{\mathbb{V}}^{n_1}_1=\mathcal{E}$.
    Define for $i=1,2$ the subspaces $\mathbb{V}_i=\hat{\mathbb{V}}_i\cap(\mathbb{R}^N_{\nu}\times 0)\subseteq\mathbb{R}^N_{\nu}$. 
    Notice that $\mathbb{V}_i\subseteq\text{span}(\alpha_{i+1})^\perp$ (index modulo $2$). Hence $\mathbb{V}_1$ and $\mathbb{V}_2$ are orthogonal definite subspaces. 
    Define then $\mathbb{V}^l:=(\mathbb{V}_1\oplus\mathbb{V}_2)^{\perp}$. Thus, $\text{span}(\alpha_i)\subseteq\mathbb{V}_i\oplus\mathbb{V}^l$ and $$l=\dim(\mathbb{V}^l)=N-\dim(\mathbb{V}_1)-\dim(\mathbb{V}_2)\leq N-(n_1-k)-(n_2-k)=N+2k-(N+k-\dim(\mathcal{E}))\leq k.$$
\end{proof}

\subsection{Diagonalizable Codazzi tensors}
The main goal of this subsection will be to prove Proposition \ref{demostracion Bajar coordenadas} which states when diagonalizing directions of a Codazzi tensor descend as coordinate vectors to the leaf space of the nullity distribution of such tensor. This result is presented in a general context since it has independent interest. This result was present in the literature when the leaf space has dimension 2 in several works, see for example \cite{FF}, \cite{DFT2}, \cite{DF2} and \cite{DFT}.

\text{ }

\begin{defn}
\normalfont Consider a real vector bundle $F\rightarrow M^n$ with a connection $\nabla=\nabla^F$. We say that a bilinear symmetric tensor $\beta:TM\times TM\rightarrow F$ {\it satisfies Codazzi equation} if $\forall X,Y,Z\in TM$
\begin{equation}\label{definicion de ser tensor Codazzi}
    (\nabla_X\beta)(Y,Z)=(\nabla_Y\beta)(X,Z).
\end{equation}
We denote $\Delta=\Delta_\beta$ the nullity of $\beta$. 
\end{defn}

\begin{remark}
    Codazzi equation implies that the nullity is in fact a totally geodesic distribution on an open dense subset of $M^n$, where $\Delta$ has constant dimension on each connected component. We assume that this is the case and that $L^l=M^n/\Delta$ is smooth of dimension $l$.
\end{remark}
We define the splitting tensor of $\Delta$ in the same way as in Definition \ref{definicion spliting tensor}, but for $\Delta$ instead of $\Gamma$. Here we also denote $X^h$ for the projection of $X\in TM$ on $\Delta^\perp$. In this context, equation (\ref{beta y CT}) is also valid for the splitting tensor of $\Delta$ since $\beta$ satisfies Codazzi equation.
    
\begin{defn}
    \normalfont Suppose that $\beta:TM\times TM\rightarrow F$ is a bilinear tensor with $l=\dim(\Delta^\perp)$ and that it is diagonalizable by the smooth frame $X_1,X_2,\ldots,X_{l}\in\Gamma(\Delta^\perp_\mathbb{C})$ with $\overline{X_{2j-1}}=X_{2j}$ for $j\leq s$ and $\overline{X_{j}}=X_{j}$ for $j>2s$ for some $s$. We say that it diagonalizes {\it strongly} if for every non-empty subset $S\subseteq\{1,\ldots,l\}$ with $\#S\leq 3$, the set $\{\beta(X_i,X_i)\}_{i\in S}$ is pointwise $\mathbb{C}$-linearly independent. 
\end{defn}
    As before, we will denote $\overline{j}$ the index associated to $j$ such that $\overline{X_j}=X_{\overline{j}}$.
\begin{prop}\label{demostracion Bajar coordenadas} 
    Let $\beta$ be a bilinear tensor satisfying Codazzi equation and
    \begin{equation}\label{R(X,Y)T=0}
        (R(X,T)S)^h=0\quad\forall T,S\in\Delta,\, X\in TM.
    \end{equation}
    Assume that $\beta$ strongly diagonalizes by $X_1,\ldots,X_l$. Then there exist $f_i:M^n\rightarrow\mathbb{C}$ satisfying $\overline{f_{j}}=f_{\overline{j}}$ $\forall j$ and (local) coordinates $(z_1,\ldots,z_s,w_{2s+1},\ldots,w_l)\in\mathbb{C}^s\times\mathbb{R}^r$ for $L^l$, such that for $Z_i=f_i X_i$, we have $\partial_{u_i}\circ\pi =\pi_*\circ Z_i$ 
    where $(u_0,\ldots,u_l)=(z_1,\overline{z_1},\ldots,z_s,\overline{z_s}, w_{2s+1},\ldots, w_l)$ and $\pi:M^n\rightarrow L^l$ is the canonical projection.
\end{prop}

\begin{proof}
    Given any vector $Y$, we write $Y^i$ for the component of $Y^h$ with respect to $X_i$, that is, $Y^h=\sum_i Y^i X_i$. 
    
    By (\ref{beta y CT}) for $X=X_i$ and $Y=X_j$ with $i\neq j$, we have that 
    $$(C_T X_i)^j\beta(X_j,X_j)=(C_T X_j)^i\beta(X_i,X_i).$$ 
    Since $\beta$ diagonalizes strongly, the last equation implies that there exist 1-forms $\lambda_i:\Delta_{\mathbb{C}}\rightarrow\mathbb{C}$ such that $C_TX_i=\lambda_i(T)X_i$.
    
    Using Codazzi equation (\ref{definicion de ser tensor Codazzi}) for $X=T$, $Y=X_i$ and $Z=X_j$, we get
    $$(\nabla_TX_i)^j\beta(X_j,X_j)+(\nabla_TX_j)^i\beta(X_i,X_i)=0.$$
    Hence, $\nabla_TX_i=a_i(T)X_i$ for some 1-forms $a_i$. First, we claim that we can assume that $a_i=0$, to simplify computations. 
    
    Equation (\ref{R(X,Y)T=0}) can be expressed in terms of the splitting tensor as
    $$\nabla_T C_S=C_S C_T+C_{\nabla_T S}.$$
    Thus
    \begin{equation}\label{dlambda_i=0}
        0=(\nabla_T C_S(X_i)-C_SC_T(X_i)-C_{\nabla_T S}(X_i))-(\nabla_S C_T(X_i)-C_TC_S(X_i)-C_{\nabla_S T}(X_i))=d\lambda_i(T,S)X_i.
    \end{equation}
    Using Jacobi identity for $T,S$ and $X_i$, and analyzing the vertical component, we get that
    \begin{equation*}
        da_i(T,S)+d\lambda_i(T,S)=0.
    \end{equation*}
    We get from (\ref{dlambda_i=0}) that $da_i(T,S)=0$. We integrate the 1-forms $a_i$ along the nullity leaves giving arbitrary values along a transversal submanifold. This defines functions $r_i$ such that $dr_i(T)=a_i(T)$. Notice that we can do this in a way that $\overline{r_i}=r_{\overline{i}}$. By replacing $X_i$ with $e^{-r_i}X_i$, we can suppose that $\nabla_T X_i=0$ for all $T\in\Delta$, as we claimed.
    
    Codazzi equation (\ref{definicion de ser tensor Codazzi}) for $X=X_i$, $Y=X_j$ and $Z=X_k$ with $i\neq j\neq k\neq i$ gives
    $$-([X_i,X_j])^k\beta(X_k,X_k)=-(\nabla_{X_i}X_k)^j\beta(X_j,X_j)-(\nabla_{X_j}X_k)^i\beta(X_i,X_i).$$
    As $\beta$ diagonalizes strongly, we get that $(\nabla_{X_i}X_j)^k=0$ for all distinct indices. Then, there exist $a_{i}^j,b_{j}^i,r_{i}^j:M^n\rightarrow \mathbb{C}$, $1\leq i\neq j\leq l$, such that
    \begin{equation}\label{ai bi ri}
        \nabla_{X_i}X_j+a_{i}^jX_i-b_{j}^iX_j\in\Delta_\mathbb{C} \quad\text{and}\quad [X_i,X_j]+r_{i}^jX_i-r_{j}^iX_j\in\Delta_\mathbb{C}.
    \end{equation}
    Clearly, $r_{i}^j=a_{i}^j+b_{i}^j$. 
    
    As in Proposition 10 of \cite{DFT}, to project $Z_i$ to $L^l$ we need that $[Z_i,T]\in\Delta_\mathbb{C}$ for all $T\in\Delta_\mathbb{C}$, and to be a local coordinate system we also need that $[Z_i,Z_j]\in\Delta_\mathbb{C}$ for any $i,j$. 
    Write $f_i=e^{g_i}$. The first condition is equivalent to $T(g_i)=-\lambda_i(T)$, while the second one is equivalent to $X_i(g_j)=-r^i_j$ and $X_j(g_i)=-r^j_i$. To find such functions, consider the $\mathbb{C}$-linear 1-form $\hat{\sigma}_i:\text{span}_\mathbb{C}\{\Delta,X_j\}_{j\neq i}\rightarrow \mathbb{C}$, given by
    $$\hat{\sigma}_i(T)=-\lambda_i(T),\quad \hat{\sigma}_i(X_j)=-r^j_i.$$
    Let's prove that $\sigma_i$ is exact. We have already proved that $d\hat{\sigma}_i|_{\Delta\times\Delta}=0$ in (\ref{dlambda_i=0}). Now, we need to prove that
    \begin{equation}\label{eq:T-j}
        d\hat{\sigma}_i(T,X_j)=-T(r^j_i)+X_j(\lambda_i(T))-\lambda_i(\nabla^v_{X_j}T)+\lambda_j(T)r^j_i=0,\quad\forall j\neq i.
    \end{equation}
    By Jacobi identity for $i\neq j$
    \begin{align*}
        0&=[T,[X_i,X_j]]^h+[X_j,[T,X_i]]^h-[X_i,[T,X_j]]^h\\
        &=(\nabla_T[X_i,X_j]^h+C_T([X_i,X_j]^h)+[X_i,-\nabla_{X_j}^vT+\lambda_j(T)X_j]^h-[X_j,-\nabla_{X_i}^vT+\lambda_i(T)X_j]^h\\
        &=-T(r_i^j)X_i+T(r_j^i)X_j-\lambda_i(T)r_i^jX_i+\lambda_j(T)r_j^iX_j+\lambda_i(\nabla^v_{X_j}T)X_i+X_i(\lambda_j(T))X_j\\
        &\quad+\lambda_j(T)(-r_i^jX_i-r_j^iX_j)-\lambda_j(\nabla^v_{X_i}T)X_j+X_j(\lambda_i(T))X_i-\lambda_i(T)(-r_j^iX_j-r_i^jX_i)\\
        &=d\hat{\sigma}_i(T,X_j)X_i-d\hat{\sigma}_j(T,X_i)X_j,
    \end{align*}
    which shows (\ref{eq:T-j}). Also by Jacobi identity we have for three distinct indices that 
    \begin{align*}
        0&=\sum[X_i,[X_j,X_k]]^{h}=\sum \Big(-\lambda_i([X_j,X_k]^v)X_i+\nabla^h_{X_i}(-r^k_jX_j+r^j_kX_k)-\nabla^h_{-r^k_jX_j+r^j_kX_k}X_i\Big)\\
        &=\sum \Big(-\lambda_i([X_j,X_k]^v)X_i-X_i(r^k_j)X_j-r^k_j\nabla^h_{X_i}X_j+X_i(r^j_k)X_k+r^j_k\nabla^h_{X_i}X_k+r^k_j\nabla^h_{X_j}X_i-r^j_k\nabla^h_{X_k}X_i\Big)\\
        &=\sum\Big(-\lambda_i([X_j,X_k]^v)X_i-X_i(r^k_j)X_j-r^k_j(-a^j_iX_i+b^i_jX_j)+X_i(r^j_k)X_k+r^j_k(-a^k_iX_i+b^i_kX_k)\\
        &\quad\quad+r^k_j(-a^i_jX_j+b^j_iX_i)-r^j_k(-a^i_kX_k+b^k_iX_i)\Big)\\
        &=\sum\Big(-\lambda_i([X_j,X_k]^v)X_i+r^k_jr^j_iX_i-r^j_kr^k_iX_i-X_i(r^k_j)X_j-r^k_jr^i_jX_j+X_i(r^j_k)X_k+r^j_kr^i_k X_k\Big)\\
        &=\sum\Big(-\lambda_i([X_j,X_k]^v)X_i+r^k_jr^j_iX_i-r^j_kr^k_iX_i-X_k(r^j_i)X_i-r^j_ir^k_iX_i+X_j(r^k_i)X_i+r^k_ir^j_i X_i\Big)\\
        &=\sum \Big(\hat{\sigma}_i([X_j,X_k])+X_k(\hat{\sigma}_i(X_j))-X_j(\hat{\sigma}_i(X_k))\Big)X_i=-\sum d\hat{\sigma}_i(X_j,X_k)X_i.
    \end{align*}
    This shows $d\hat{\sigma}_i(X_j,X_k)=0$ and proves the exactness of $\hat{\sigma}_i$.
    
    For $1\leq i\leq l$, consider 
    $$\hat{\Omega}_i=\text{span}_{\mathbb{C}}\{\Delta,X_j\}_{j\neq i, \overline{i}}.$$ 
    As the $X_i$'s are the eigenvectors of the splitting tensors, by (\ref{ai bi ri}) $\hat{\Omega}_i$ is involutive, namely, it is closed with respect to the Lie bracket extended by $\mathbb{C}-$bilinearity.
    Since $\hat{\Omega}_i$ is closed with respect to conjugation of indices, this implies that $\Omega_i=\text{Re}(\hat{\Omega}_i)\subseteq TM$ is integrable in the Frobenius sense. Consider $\sigma_i=\hat{\sigma}_i|_{\Omega_i}$ which is a closed 1-form, since $\hat{\sigma}_i$ is closed. 
    Therefore, we can integrate $\sigma_i$ on $M^n$ by defining arbitrarily values along a transversal submanifold to $\Omega_i$. Thus, there exists $g_i$'s such that $dg_i|_{\Omega_i}=\sigma_i$. 
    This can be done in a way that $\overline{g_i}=g_{\overline{i}}$. 
    
    Consider then $Y_i=e^{g_i}X_i$. Those vectors satisfy that $[Y_i,T]\in\Delta_{\mathbb{C}}$ and $[Y_i,Y_j]\in\Delta_\mathbb{C}$ for any $T\in\Delta$ and $i\neq \overline{j}$. 
    Using Proposition 10 of \cite{DFT}, let $A_i\in (TL)_\mathbb{C}$ be the local frame such that $A_i\circ\pi=\pi_{*}Y_i$. They satisfy that $[A_i,A_j]=0$ for any $i\neq \overline{j}$. If there are no complex indices, we are done. Thus, suppose that this is not the case.
    
    Write $A_{2j}=U_j+iV_j$ for $j\leq s$. By (\ref{ai bi ri}), there exist $a_j,b_j:L^l\rightarrow\mathbb{R}$ such that
    \begin{equation*}
        [A_{2j-1},A_{2j}]+(a_j+ib_j)A_{2j-1}-(a_j-ib_j)A_{2j}=0,
    \end{equation*}
    which in terms of the $U_j$'s and $V_j$'s can be expressed as
    \begin{equation*}
        [U_j,V_j]+b_jU_j-a_jV_j=0.
    \end{equation*}
    For $k\neq 2j,2j-1$, from Jacobi identity using the last condition we get that
    \begin{equation}\label{dsa}
        A_k(a_j)=A_k(b_j)=0.
    \end{equation}
    Thus, there are (local) functions $\hat{a}_j,\hat{b}_j:L^l\rightarrow \mathbb{R}$ such that the frame 
    $$\{e^{\hat{a}_1}U_1,e^{\hat{b}_1}V_1,\ldots,e^{\hat{a}_s}U_s,e^{\hat{b}_s}V_s, A_{2s+1},\ldots,A_l\},$$ 
    is commutative. Then there is a local chart $(x_1,y_1,\ldots,x_s,y_s,w_{2s+1},\ldots,w_l)$ such that the canonical vectors are this frame (locally) and $\hat{a}_j=\hat{a}_j(x_j,y_j)$ $\hat{b}_j=\hat{b}_j(x_j,y_j)$  by (\ref{dsa}). 
    
    To conclude, consider on the plane $(x_j,y_j)$ the metric $g(\partial_{x_j},\partial_{x_j})=e^{2\hat{a}_j}$, $g(\partial_{y_j},\partial_{y_j})=e^{2\hat{b}_j}$ and $g(\partial_{x_j},\partial_{y_j})=0$. 
    Since all the surfaces possess isothermal charts, there are functions $p_j=p_j(x_j,y_j)$ and $q_j=q_j(x_j,y_j)$ with $(p_j,q_j)\neq (0,0)$ such that $[p_iU_i-q_iV_i,p_iV_i+q_iV_i]=0$. 
    Thus, there is a local chart $(\hat{x}_1,\hat{y}_1,\ldots,\hat{x}_s,\hat{y}_s,w_{2s+1},\ldots,w_l)$ such that $\partial_{\hat{x}_i}=p_iU_i-q_iV_i$, $\partial_{\hat{y}_i}=q_iU_i+p_iV_i$. This chart is the chart we are looking for. Define $z_j=\hat{x}_j+i\hat{y}_j$, $f_{2j}=e^{g_{2j}}(p_j+iq_j)$, $f_{2j-1}=\overline{f_{2j}}$ for $j\leq s$ and $f_k=e^{g_k}$ for $k> 2s$. 
\end{proof}

\printbibliography

IMPA – Estrada Dona Castorina, 110

22460-320, Rio de Janeiro, Brazil

{\it E-mail address}: {\tt diego.navarro.g@ug.uchile.cl}
\end{document}